\documentclass[11pt]{article}
\usepackage{amssymb,amsmath}
\usepackage{float}
\usepackage{color,fullpage}
\usepackage{cite}
\usepackage[title,titletoc]{appendix}
\usepackage{algorithm,algorithmic}
\UseRawInputEncoding
\usepackage{amsthm}
\usepackage{graphics,graphicx,epstopdf}

\usepackage{physics}

\definecolor{blue}{rgb}{0,0,0.9}
\definecolor{red}{rgb}{0.9,0,0}
\definecolor{green}{rgb}{0,0.9,0}

\newcommand{\be}{\begin{equation}}
\newcommand{\ee}{\end{equation}}
\newcommand{\ben}{\begin{enumerate}}
\newcommand{\een}{\end{enumerate}}
\newcommand{\bfr}{\begin{frame}}
\newcommand{\efr}{\end{frame}}
\newcommand{\bit}{\begin{itemize}}
\newcommand{\eit}{\end{itemize}}
\newcommand{\argmin}{\mathop{\rm argmin}}

\newcommand{\R}{\mathbb R}
\newcommand{\E}{\mathbb{E}}

\newcommand{\cC}{\mathcal C}
\newcommand{\cD}{\mathcal D}

\newcommand{\cF}{\mathcal F}

\newcommand{\cH}{\mathcal H}
\newcommand{\cI}{\mathcal I}
\newcommand{\cJ}{\mathcal J}

\newcommand{\cM}{\mathcal M}
\newcommand{\cN}{\mathcal N}

\newcommand{\cS}{\mathcal S}

\def\norm#1{\left\|#1\right\|}
\def\inner#1{\left\langle#1\right\rangle}
\def\inprod#1#2{\langle#1,\,#2\rangle}

\newtheorem{definition}{Definition}[section]
\newtheorem{theorem}{Theorem}[section]
\newtheorem{corollary}{Corollary}[section]
\newtheorem{lemma}{Lemma}[section]
\newtheorem{proposition}{Proposition}[section]
\newtheorem{remark}{Remark}[section]
\newtheorem{assumption}{Assumption}[section]

\newtheorem{example}{Example}[section]

\def\conv{{\rm conv}}

\def\dist{{\rm dist}}
\def\dom{{\rm dom}}

\begin{document}

\title{On exploration of an interior mirror descent flow for stochastic nonconvex constrained problem}
\author{
Kuangyu Ding\thanks{School of Industrial Engineering, Purdue University, West Lafayette, IN, USA,
         47906 ({\tt kuangyud@u.nus.edu}).
         }, \quad 
Kim-Chuan Toh\thanks{Department of Mathematics, and Institute of 
Operations Research and Analytics, National
         University of Singapore, 
       Singapore
         119076 ({\tt mattohkc@nus.edu.sg}).
         }
        }
\date{\today}

\maketitle

\begin{abstract}
We study a nonsmooth nonconvex optimization problem defined over nonconvex constraints, where the feasible set is given by the intersection of the closure of an open set and a smooth manifold. By endowing the open set with a Riemannian metric induced by a barrier function, we obtain a Riemannian subgradient flow formulated as a differential inclusion, which remains strictly within the interior of the feasible set. This continuous dynamical system unifies two classes of iterative optimization methods, namely the Hessian barrier method and mirror descent scheme, by revealing that these methods can be interpreted as discrete approximations of the continuous flow. We explore the long-term behavior of the trajectories generated by this dynamical system and show that the existing deficient convergence properties of the Hessian barrier and mirror descent scheme can be unifily and more insightfully interpreted through these of the continuous trajectory. For instance, the notorious spurious stationary points \cite{chen2024spurious} observed in mirror descent scheme are interpreted as stable equilibria of the dynamical system that do not correspond to real stationary points of the original optimization problem. We provide two sufficient conditions such that these spurious stationary points can be avoided. In the absence of these regularity conditions, we propose a novel random perturbation strategy that ensures the trajectory converges (subsequentially) to an approximate stationary point. Building on these insights, we introduce two iterative Riemannian subgradient methods, form of interior point methods, that generalizes the existing Hessian barrier method and mirror descent scheme for solving nonsmooth nonconvex optimization problems. 
\medskip

\end{abstract}

\section{Introduction}

This paper investigates the convergence of trajectories generated by a continuous Hessian-Riemannian dynamical system under generic nonconvex constraints, as well as the convergence properties of its associated iterative optimization algorithms. Specifically, we focus on the generic constrained problem
\begin{equation}
\begin{aligned}
\min_{x\in\R^n}\;&f(x)\\
\text{s.t.}\;&x\in\overline C\cap\cM,
\end{aligned}
\label{prob:PCP}
\end{equation}
where \(C\subset\R^n\) is an open and convex set, \(f\colon\R^n\to\R\) is locally Lipschitz continuous, and
\[
\cM \;=\;\{\,x\in\R^n : c(x)=0\,\}
\]
is a smooth manifold, where the linear independence constraint qualification (LICQ) holds at every point of \(\cM\). A prototypical special case arises when \(C=\R^n_{++}=\{x_i>0\}\) and \(\cM=\{x:Ax=b\}\), in which case \(\cM\) is an affine subspace (also denoted by \(L\)). Problems of this form with linear constraints are fundamental in optimization and appear in countless applications.  Beyond the affine-linear setting, formulations of the type \eqref{prob:PCP} with a nontrivial manifold \(\cM\) are likewise ubiquitous in practice, for example in nonnegative PCA \cite{zass2006nonnegative,jiang2023exact}, sparse PCA \cite{zou2006sparse}, and low-rank matrix completion \cite{vandereycken2013low,cambier2016robust}.

When \(\cM\) is an affine subspace, a full survey of solution methods for \eqref{prob:PCP} is beyond the scope of this paper, but we note that the literature includes conditional gradient descent (Frank-Wolfe) \cite{frank1956algorithm,jaggi2013revisiting}, interior-point methods \cite{karmarkar1984new,nesterov1994interior,wright1997primal}, penalty methods \cite{fiacco1968nonlinear,xiao2025exact} including augmented Lagrangian \cite{powell1969method,rockafellar1973dual} and active-set schemes \cite{byrd1999active}, as well as mirror descent (Bregman-type) scheme \cite{nemirovskij1983problem,bauschke2017descent,lu2018relatively,bolte2018first}.  In this work, we focus on two representative discretizations, namely the mirror descent scheme and the recently proposed Hessian-barrier method \cite{bomze2019hessian,dvurechensky2025hessian}, both of which arise from the same underlying Riemannian gradient flow on the polyhedral feasible set:
\[
\dot x \;=\; -\,\mathrm{P}_x\,H(x)^{-1}\,\nabla f(x), 
\quad x(0)\in C\cap L,
\]
first studied in \cite{bolte2003barrier,alvarez2004hessian,attouch2004singular}.  Here \(f\) is assumed differentiable, the Riemannian metric is induced by the Hessian of a barrier function \(\phi\) for \(\overline C\), i.e.\ \(H(x)=\nabla^2\phi(x)\), and \(\mathrm{P}_x\) denotes projection onto the affine subspace \(L\) under this metric.

The Hessian-barrier method discretizes the above Riemannian gradient flow and was first proposed in \cite{bomze2019hessian} for nonconvex linearly constrained problems. It has more recently been applied to nonconvex conic programming \cite{dvurechensky2025hessian}.  Despite its broad applicability, its convergence guarantees are generally weaker than those of classical first-order methods: it only converges to a scaled KKT point, in which the feasibility of the dual slack variables may fail, unless additional regularity (e.g.\ convexity of \(f\), isolation of limit points, or strict complementarity) is imposed.  In the special case where \(\phi\) is the entropy barrier, the method reduces to a replicator-type descent algorithm \cite{tseng2011first}, which likewise exhibits deficient convergence behavior.

Mirror descent can be viewed as an alternative discretization of the same Riemannian Hessian-barrier flow.  First introduced by Nemirovski and Yudin \cite{nemirovskij1983problem} for nonsmooth convex optimization, it has become a foundational scheme in both convex \cite{benamou2015iterative,bauschke2017descent,dragomir2021quartic,dragomir2021fast,yang2022bregman,chu2023efficient,khoo2025bregman} and nonconvex \cite{bolte2018first,mukkamala2019bregman,ding2023nonconvex,zhang2024stochastic,ding2024stochastic} settings.  However, mirror descent is known to suffer from several pathological convergence behaviors: it is unable to, in general, achieve accelerated \(O(1/k^2)\) rates in smooth convex problems \cite{dragomir2021optimal}; commonly used Bregman stationarity measures may fail to capture true stationarity in nonconvex unconstrained problems \cite{zhang2024stochastic}; and iterates can become trapped at spurious stationary points, escaping them arbitrarily slowly \cite{chen2024spurious}.  Whether mirror descent iterates for general nonconvex constrained problems converge to the true stationary points remains an {open question}, largely due to the blowing up behavior of the kernel (distance generating) function near the boundary of its domain.

\begin{figure*}[tb]
		\begin{center}
			\setlength{\tabcolsep}{0.0pt}  
			\scalebox{1}{\begin{tabular}{ccc}
					\includegraphics[width=0.33\linewidth]{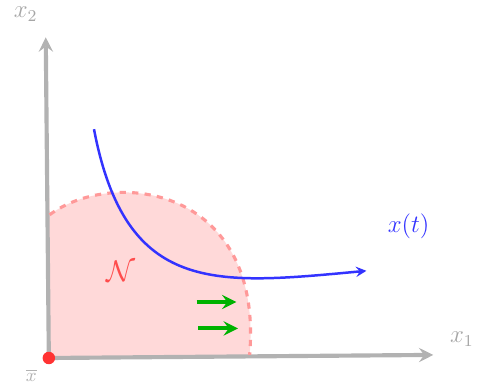}&
					\includegraphics[width=0.33\linewidth]{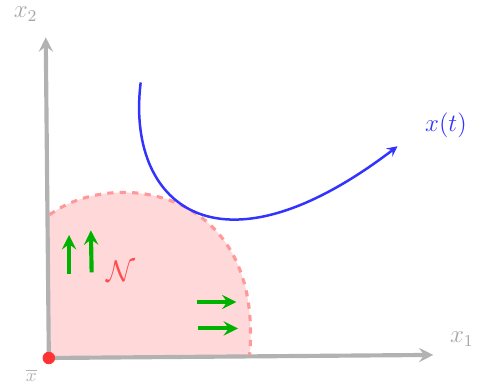}&
					\includegraphics[width=0.33\linewidth]{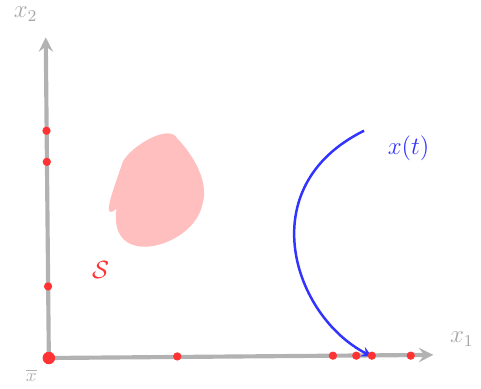}\\
					\multicolumn{1}{c}{\footnotesize{(a) Repelling from one direction}} &  \multicolumn{1}{c}{\footnotesize{(b) Repelling from all directions}}&
					\multicolumn{1}{c}{\footnotesize{(c) Isolated stable set over boundary}}                  
			\end{tabular}}
            		\caption{Properties of escaping from spurious stationary points. $\bar x$ represents a spurious stationary point, and $\mathcal{N}$ is a neighborhood of $\bar x$. (a) corresponds to Proposition \ref{prop:exiting}. (b) corresponds to Theorem \ref{thm:completmentarity}. (c) corresponds to Theorem \ref{thm:isolated}.} \label{fig:escaping}
		\end{center}
	\end{figure*}

In this work, we bring to the community's attention both the deficient convergence properties of the Hessian barrier method and the spurious stationary point phenomenon of mirror descent.  We attempt to explore the open question by studying a unified continuous dynamical system underlying both algorithms and show that their pathological behaviors coincide each being governed by the same equilibria (the \emph{stable set}, defined in Section 2) of the flow. Specifically, we investigate the behavior of the continuous trajectory near spurious stationary points. For instance, under the complementarity condition (as established in Theorem \ref{thm:completmentarity}) and the isolated condition (as shown in Theorem \ref{thm:isolated}), illustrated in Figure \ref{fig:escaping}.(c), the trajectory consistently converges to a true stationary point. When the complementarity condition is strictly violated (as demonstrated in Theorem \ref{thm:completmentarity}), the trajectory is repelled in all directions, ensuring that spurious stationary points are never limit points. In the general case, without specific regularity conditions, we prove that the trajectory is repelled in at least one direction, exiting any neighborhood of a spurious stationary point in finite time, as depicted in Figure \ref{fig:escaping}.(a) and formalized in Proposition \ref{prop:exiting}. Moreover, our framework extends the Riemannian Hessian-barrier gradient flow of \cite{alvarez2004hessian,bomze2019hessian} by allowing \(\cM\) to be a general manifold (rather than an affine subspace) and by admitting nonsmooth objectives \(f\). 

We now summarize our contributions:
\begin{enumerate}
    \item We propose a novel continuous dynamical system (in the form of a differential inclusion) for the constrained Problem \eqref{prob:PCP}, by leveraging the manifold structure of the feasible set. This differential inclusion is inspired by the Riemannian Hessian-barrier gradient flow for linearly constrained problems. By discretizing this inclusion, we derive two new optimization algorithms that are interior point methods in nature, in the sense that all iterates lie in the relative interior: \(\cM\cap C\).
    \item We interpret the pathological behaviors of both the Hessian barrier method and the mirror descent scheme within a unified framework. In particular, we show that the spurious stationary points observed in both algorithms correspond precisely to the points in the stable set of the continuous differential inclusion, but not the true stationary points of the original optimization problem. We analyze the escape dynamics near these spurious points, revealing a repelling phenomenon. As a byproduct, if a trajectory converges, it must converge to a true stationary point of the optimization problem. We further provide two sufficient conditions that guarantee avoidance of spurious stationary points, hence subsequential convergence to true stationary points. In the absence of these conditions, we introduce a novel perturbation strategy that ensures convergence to a true stationary point of a perturbed problem, that is the nearly stationary of the original problem in many instances.
    \item Building on the continuous trajectory analysis and leveraging stochastic approximation theory, we establish convergence results for the induced discrete algorithms. Under either of the two sufficient conditions, the discrete iterates subsequentially converge to true stationary points of \eqref{prob:PCP}. Without these conditions, our perturbation strategy ensures convergence to a stationary point of the perturbed problem, which approximates a nearly stationary point of the original problem in many instances.
\end{enumerate}

The remainder of the paper is organized as follows. In Section 2, we review notations, nonsmooth analysis, the Riemannian Hessian-barrier subgradient flow for linearly constrained problems, and stochastic approximation. Section 3 introduces the unified differential inclusion for Problem \eqref{prob:PCP}, establishes its connection to spurious stationary points, and studies the behavior of its trajectories. Section 4 presents avoidance properties of spurious stationary points. Section 5 derives two new optimization algorithms by discretizing the differential inclusion and provides their convergence analysis and discussion. A brief conclusion is provided in Section 6.

\section{Preliminaries}
\subsection{Notations}
Given a proper and lower semicontinuous function $f:\mathbb{R}^n\rightarrow\overline{\mathbb{R}} := (-\infty,\infty]$, we denote its domain as $\dom\,f=\{x:f(x)<\infty\}$. The Fenchel conjugate function of $f$ is defined as $f^*(y):=\sup\{\langle x,y\rangle - f(x):x\in\mathbb{R}^n\}$. For a set $\mathcal{S}\subset\mathbb{R}^n$, we use ${\overline{\mathcal{S}}}$ to denote its closure, ${\rm int}\,\mathcal{S}$ to denote the set of its interior points, and ${\rm conv}(\cS)$ to denote its convex hull. A function $f:\cS\rightarrow\R$ is said to be of class ${\cC}^k(\mathcal{S})$ if it is $k$ times differentiable and the $k$-th derivative is continuous on $\mathcal{S}$. When there is no ambiguity regarding the domain, we simply use the notation ${\cC}^k$. We use $\|\cdot\|$ to denote the Euclidean norm for vectors and the Frobenius matrix norm for matrices. The $n$-dimensional ball with radius $\delta$ is denoted by $\mathbb{B}^n_\delta$. For any given set-valued mapping $\cH: \R^n \rightrightarrows \R^n$ and any constant $\delta \geq 0$,  the $\delta$-expansion of $\cH$, denoted as $\cH^{\delta}$,  is defined as
\begin{equation*}
    \cH^{\delta}(x) := \{ w \in \R^n: \exists z \in \mathbb{B}_{\delta}(x), \, \mathrm{dist}(w, {\cH}(z))\leq \delta \}.
\end{equation*}

Let $(\Omega,\mathcal{F},\mathbb{P})$ be a probability space. Consider a stochastic process $\{\xi_k\}_{k\geq0}$ and a filtration $\{\mathcal{F}_k\}_{k\geq0}$, where $\mathcal{F}_k$ is defined by the $\sigma$-algebra $\mathcal{F}_k:=\sigma(\xi_0,\ldots,\xi_k)$ on $\Omega$, the conditional expectation is denoted as $\mathbb{E}[\cdot|\mathcal{F}_k]$.

\subsection{Nonsmooth analysis}

In this subsection, we recall some concepts from nonsmooth analysis, following \cite{clarke1990optimization,davis2020stochastic,bolte2021conservative}.

\begin{definition}[Clarke subdifferential {\cite{clarke1990optimization}}]
\label{Defin_Subdifferential}
Let \(f\colon\R^n\to\R\) be locally Lipschitz continuous. For any \(x\in\R^n\), the \emph{Clarke subdifferential} of \(f\) at \(x\) is
\[
\partial f(x)
\;=\;
\conv\Bigl\{\lim_{k\to\infty}\nabla f(x_k)\colon x_k\to x,\;f\text{ is differentiable at each }x_k\Bigr\}.
\]
\end{definition}

The next definition introduces the main function class for our objectives.

\begin{definition}[Path-differentiable {\cite{davis2020stochastic,bolte2021conservative}}]
\label{def:path-differentiable}
A locally Lipschitz function \(f\colon\R^n\to\R\) is \emph{path-differentiable} if it satisfies the chain rule: for every absolutely continuous curve \(z\colon[0,+\infty)\to\R^n\), the composition \(f\circ z\) is absolutely continuous and, for almost every \(t\ge0\),
\[
\frac{d}{dt}\bigl(f(z(t))\bigr)
\;=\;
\bigl\langle v,\,\dot z(t)\bigr\rangle
\quad\text{for all }v\in\partial f\bigl(z(t)\bigr).
\]
\end{definition}

In this work, we focus on path-differentiable functions. This class is broad enough to model objectives arising in many real-world applications. Indeed, any Clarke-regular function is path‐differentiable \cite[Section 5.1]{davis2020stochastic}. More generally, any function admitting a Whitney \(\cC^s\) stratification (for \(s\ge1\)) is path-differentiable \cite[Definition 5.6]{davis2020stochastic}, and any function definable in an \(o\)-minimal structure inherits such a stratification \cite[Definition 5.10]{davis2020stochastic}. Furthermore, one can extend beyond the Clarke subdifferential via the notion of a {conservative field}, as introduced in \cite{bolte2021conservative} and further studied in \cite{castera2021inertial,xiao2023adam,le2023nonsmooth,ding2023adam,xiao2024developing}; conservative fields often better capture "subgradient" produced by automatic differentiation. Since this extension lies beyond our present focus, we do not discuss it further.

\subsection{Riemannian Hessian-barrier gradient flow over linear constraints}
In this subsection, we review the Riemannian Hessian–barrier gradient flow for optimization problems with linear constraints, i.e. $\cM=L:=\qty{x\in\R^n:Ax=b}$, $C=\qty{x\in\R^n:x\geq0}$, where $A\in\R^{m\times n}$. The results extend the smooth case analysis of \cite{alvarez2004hessian,bomze2019hessian} to the nonsmooth setting. Additional basic properties are also provided, which will be used in the proofs of our main results. 

\begin{definition}\label{Bregman-distance-def}
(Legendre function and Bregman distance).  A function $\phi:C\rightarrow\R$ is called a Legendre function over $C$, if ${\rm int\,dom}(\phi)=C$, strictly convex in $C$, and $\phi\in\cC^1(C)$. Moreover, for any $x_k$ approaching the boundary of $C$, $\norm{\nabla\phi(x)}\to\infty$. The Bregman distance \cite{bregman1967relaxation} generated by $\phi$ is denoted as $\cD_\phi(x,y):{\rm dom}(\phi)\times{\rm dom}(\nabla\phi)\rightarrow[0,+\infty)$, where
\[
\cD_\phi(x,y)=\phi(x)-\phi(y)-\inprod{\nabla\phi(y)}{x-y}.
\]
\end{definition}

Let $H(x)=\nabla^2\phi(x)$, which is positive‐definite on \(C\). Define the local norm $\norm{z}_x:=\sqrt{\inner{\nabla^2\phi(x)z,z}}.$ Given a point $x$ and $d(x)\in\partial f(x)$, the search direction is determined by the following projection problem under a Riemannian metric $\inner{\cdot,\cdot}_{\nabla^2\phi(x)}$:
\begin{equation}
\label{eq:v(x)}
\begin{aligned}
\min_{v\in\R^n,Av=0}\;&\inner{d(x),v}+\frac{1}{2}\norm{v}_x^2
\end{aligned}
\end{equation}
The primal and dual solutions of \eqref{eq:v(x)} is given by 
\[
\begin{aligned}
v(x)&=-{\rm P}_xH(x)^{-1}d(x)\\
y(x)&=(AH(x)^{-1}A^T)^{-1}AH(x)^{-1}d(x),
\end{aligned}
\]
where ${\rm P}_x$ is the Riemannian projection onto the linear space $L_0:=\{z:Az=0\}$, i.e., the solution of the following projection problem:
\[
\begin{aligned}
\min_{z\in\R^n,Az=0}\;&\frac{1}{2}\norm{z-u}_x^2\\
\end{aligned}
\]
which has the closed form as follows 
\[
{\rm P}_x=I-H(x)^{-1}A^T(AH(x)^{-1}A^T)^{-1}A.
\]
Based on the optimality condition of \eqref{eq:v(x)}, we have
\[
d(x)+H(x)v(x)-A^Ty(x)=0,
\]
Let $v(x)=-H(x)^{-1}s(x)$, which is considered the dual slack variable in the mirror setting. So we have $s(x)=d(x)-A^Ty(x)$. 

Based on the notation above, we consider the following Riemannian Hessian-barrier gradient flow based on the search direction $v(x)$:
\begin{equation}
\label{eq:RGF1}
\dot x\in -{\rm P}_xH(x)^{-1}\partial f(x),\;x(0)\in C\cap L.
\end{equation}
\eqref{eq:RGF1} is equivalent to that there exists $d(x)\in\partial f(x)$ such that $\dot x$ is the optimal solution of  
\[
\begin{aligned}
\min_{Az=0}\;&\frac{1}{2}\norm{z+H(x)^{-1}d(x)}_x^2.
\end{aligned}
\]
Since the above problem is a strongly convex problem, it is equivalent to that $\dot x$ satisfies the KKT condition, i.e. there exists $y\in\R^m$, such that 
\[
\begin{aligned}
H(x)\left(\dot x+H(x)^{-1}d(x)\right)-A^Ty&=0\\
A\dot x&=0.
\end{aligned}
\]
This equivalent to 
\begin{equation}
\begin{aligned}
\frac{d}{dt}\partial\phi(x(t))\in-\nabla f(x(t))+N_{L},\;x(0)\in C\cap L.  
\end{aligned}
\label{eq:RGF2}
\tag{RGF2}
\end{equation}
We can verify that the two sets are equivalent by the optimality conditions of the definition of ${\rm P}_x$.
\begin{proposition}
\label{prop:equiv_stable}
$\{x\in L\cap C:0\in H(x)^{-1}(\partial f(x)+N_L)\}=\{x\in L\cap C:0\in {\rm P}_xH(x)^{-1}\partial f(x)\}$.
\end{proposition}

\begin{proof}
Let \(z\in L\cap C\). For any $z\in\{x:0\in {\rm P}_xH(x)^{-1}\partial f(x)\}$, there exists $d(z)\in\partial f(z)$, such that 
\[
0=\argmin\qty{\frac{1}{2}\norm{u-H(z)^{-1}d(z)}_x^2:Au=0}.
\]
 By optimality condition of above optimization problem, there exists $y\in\R^m$ such that 
\[
H(z)(0-H(z)^{-1}(x)d(z))+A^Ty=0.
\]
Thus, we have $0=H(z)^{-1}(d(z)-A^Ty)\in H(z)^{-1}(\partial f(z)-A^Ty)$. 

For any $z\in\{x:0\in H(x)^{-1}(\partial f(x)+N_L)\}$, there exist $d(x)\in\partial f(x)$ and $y\in\R^m$, such that $0=H(z)^{-1}(d(z)+A^Ty)$. Multiply ${\rm P}_z$ for both sides, we have 
\[
\begin{aligned}
0=&{\rm P}_z(H(z)^{-1}d(z))+{\rm P}_z(H(z)^{-1}A^Ty)\\
=&{\rm P}_z(H(z)^{-1}d(z))+\qty(I-H(z)^{-1}A^T(AH(z)^{-1}A^T)^{-1}A)H(z)^{-1}A^Ty\\
=&{\rm P}_z(H(z)^{-1}d(z)).
\end{aligned}
\]
Thus, $0\in {\rm P}_zH(z)^{-1}\partial f(z)$. This completes the proof.
\end{proof}

\begin{remark}[Extension to the boundary] 
In many applications, the operator ${\rm P}_x\,H(x)^{-1}$ admits a continuous extension from \(L\cap C\) to the closure \(L\cap\overline C\).  Intuitively, as \(x\) approaches the boundary of \(C\), the barrier \(\phi(x)\) diverges to \(+\infty\), causing its Hessian inverse \(\nabla^2\phi(x)^{-1}\) to remain well-controlled (even flattening out), rather than blowing up. Under this extension hypothesis and by Proposition~\ref{prop:equiv_stable}, the equivalence naturally carries over to the boundary, yielding
\[
\{\,x\in L\cap\overline C : 0\in {\rm P}_x\,H(x)^{-1}\,\partial f(x)\}
\;:=\;
\{\,x\in L\cap\overline C : 0\in H(x)^{-1}\bigl(\partial f(x)+N_L\bigr)\}.
\]
Throughout this paper, we assume that \({\rm P}_x\,\nabla^2\phi(x)^{-1}\) extends continuously to \(\cM\cap\overline C\).
\end{remark}

\subsection{Stochastic approximation in nonsmooth optimization}

In this subsection, we review  some fundamental concepts from stochastic approximation theory that are essential for the proofs in this paper.  The concepts discussed here are mainly adopted from \cite{benaim2005stochastic}. we refer the reader to \cite{benaim2006stochastic,borkar2009stochastic,davis2020stochastic} for more details on the stochastic approximation technique.  We begin by defining the trajectory of a differential inclusion, its stable set, and the associated Lyapunov functions, which will play a central role in our subsequent convergence analysis.

\begin{definition}
For any graph-closed set-valued mapping $\cH: \R^n \rightrightarrows \R^n$,  we say that an absolutely continuous path $x(t)$ in $\R^n$ is a trajectory of the differential inclusion 
\begin{equation}
    \label{Eq_def_DI}
    \frac{\mathrm{d} x}{\mathrm{d}t} \in -\cH(x),
\end{equation}
with initial point $x_0$ if $x(0) = x_0$, and $\dot{x}(t) \in \cH(x(t))$ holds for almost every $t\geq 0$. The $\omega$-limit set for $x(t)$ starting at $x(0)=x_0$ is defined as
\[
\omega(x_0):=\qty{y:\lim_{k\to\infty}x(t_k)=y,\;x(0)=x_0}.
\]
\end{definition}

\begin{definition}
    \label{Defin_Lyapunov_function}
     A continuous function $\Psi:\R^n \to \R$ is a Lyapunov function for the differential inclusion \eqref{Eq_def_DI}, with respect to $\cH^{-1}\{0\}$, called  stable set, if it satisfies the following conditions:
    \begin{enumerate}
        \item For any trajectory $\gamma$ of \eqref{Eq_def_DI} with 
        $\gamma(0) \in \R^n$ , it holds that $\phi(\gamma(t)) \leq \phi(\gamma(0))$ for any $t\geq0$.
        \item For any trajectory $\gamma$ of \eqref{Eq_def_DI} with 
        $\gamma(0) \in \R^n$ with $\gamma(0) \notin \cH^{-1}\{0\}$, there exists a constant $T>0$ such that {$\phi(\gamma(T)) < \sup_{t\in[0,T]}\phi(\gamma(t)) =\phi(\gamma(0))$}.
    \end{enumerate}
\end{definition}

The corresponding discrete scheme of \eqref{Eq_def_DI} is given by
\begin{equation}
   \label{Eq:general_iterative}  x_{k+1}=x_k-\eta_k(d_k+\xi_k),
\end{equation}
where $d_k$ is an evaluation of $\cH(x_k)$ with possible inexactness, and $\xi_k$ is the stochastic noise. For a positive sequence $\{\eta_k\}$, we define $\lambda_\eta(0):=0$, $\lambda_\eta(k):=\sum_{i=0}^{k-1}\eta_i$ for $k\geq1$, and $\Lambda_\eta(t):=\sup\{k:\lambda_\eta(k)\leq t\}$. In other word, $\Lambda_\eta(t)=k$ if and only if $\lambda_\eta(k)\leq t<\lambda_\eta(k+1)$. To establish the convergence of the discrete sequence $\{x_k\}$, a key idea is to show that the linear interpolation of the sequence $\{x_k\}$ defined by
\begin{equation}
{z(t)}:=x_k+\frac{t-\lambda_\eta(k)}{\eta_k}(x_{k+1}-x_k),\;t\in[\lambda_\eta(k),\lambda_\eta(k+1))
\label{Eq:x(t)_def}
\end{equation}
is a perturbed solution \cite{benaim2005stochastic} to the associated differential inclusion, which is guaranteed by the following assumptions.
\begin{assumption}
\label{Assumption: DI}
    \begin{enumerate}
        \item The sequences $\{x_k\}$ and $\{d_k\}$ are uniformly bounded.
        \item The stepsize $\{\eta_k\}$ satisfies $\sum_{k=0}^\infty\eta_k=\infty$ and $\lim_{k\rightarrow\infty}\eta_k=0$.
        \item For any $T>0$, the noise sequence $\{\xi_k\}$ satisfies 
        \begin{equation*}
            \lim_{s\rightarrow\infty}\sup_{s\leq i\leq\Lambda_\eta(\lambda_\eta(s)+T)}\norm{\sum_{k=s}^i\eta_k\xi_k}=0.
        \end{equation*}       
        \item The set-valued mapping $\mathcal{H}$ has a closed graph. Additionally, for any unbounded increasing sequence $\{k_j\}$ such that $\{x_{k_j}\}$ converges to $\bar x$, it holds that 
        \begin{equation*}
            \lim_{N\rightarrow\infty}{\rm dist}\left(\frac{1}{N}\sum_{j=1}^Nd_{k_j},\mathcal{H}(\bar x)\right)=0.
        \end{equation*}
    \end{enumerate}
\end{assumption}

\begin{assumption}\label{assumption_Sard_Lyapunov}
There exists a Lyapunov function $\Psi:\R^n\to\R$ for \eqref{Eq_def_DI}, such that the set $\{\Psi(x): \text{ for $x$ such that $0\in\cH(x)$} \}$ has empty interior in $\R$.
\end{assumption}

The following theorem summarizes the results in \cite{duchi2018stochastic,davis2020stochastic}, which establishes the convergence of $\{x_k\}$ generated by \eqref{Eq:general_iterative}. We omit the proof for brevity.
\begin{theorem}\label{convergence-thm-func-val}
Suppose Assumptions \ref{Assumption: DI} and \ref{assumption_Sard_Lyapunov} hold. Then any limit point of $\{x_k\}$ lies in $\mathcal{H}^{-1}(0)$ and the function values $\{\Psi(x_k)\}_{k\geq1}$ converge.
\end{theorem}

\section{Interior Riemannian  subgradient flow}

In this section, we propose an interior Riemannian subgradient flow on the manifold \(\cM\cap C\), extending the classical Hessian-barrier dynamical system to embedded smooth manifolds.  Since \(C\subset\R^n\) is an open convex set, it is itself a smooth manifold.  To ensure trajectories stay in the interior of \(C\), we select a barrier function \(\phi\) on \(\overline C\) and equip \(C\) with the Riemannian metric $\langle u,v\rangle_x \;=\; u^\top\nabla^2\phi(x)\,v,$ so that the unconstrained Riemannian subgradient at \(x\in C\) is \(\nabla^2\phi(x)^{-1}\partial f(x)\). Because \(\cM\) is a smooth embedded submanifold of \(\R^n\), the intersection \(\cM\cap C\) is also a smooth manifold.  Restricting the metric \(\langle\cdot,\cdot\rangle_x\) to \(\cM\cap C\) and projecting onto the tangent space \(T_x\cM\) yields the restricted Riemannian subgradient ${\rm P}_{T_x\cM}\bigl(\nabla^2\phi(x)^{-1}\partial f(x)\bigr),$ where \({\rm P}_{T_x\cM}\) is the projection under \(\langle\cdot,\cdot\rangle_x\).  The resulting differential inclusion is that for almost any $t>0$:
\begin{equation}
\label{eq:DI_manifold}
\dot x(t)\;\in\;-\,{\rm P}_{T_{x(t)}\cM}\,\nabla^2\phi\bigl(x(t)\bigr)^{-1}\,\partial f\bigl(x(t)\bigr),
\quad x(0)\in \cM\cap C.
\end{equation}
Here, \(x(0)\) lies in the relative interior of the feasible set \(\cM\cap\overline C\). 

When \(\cM\) is an affine space and \(f\) is differentiable, the differential inclusion \eqref{eq:DI_manifold} reduces to the Riemannian Hessian-barrier gradient flow originally introduced by  \cite{alvarez2004hessian}, and subsequently studied in the context of conic programming by  \cite{bomze2019hessian,dvurechensky2025hessian}.  Later, we will show that this unified continuous dynamical system captures the asymptotic behaviors of both Hessian-barrier methods and the mirror descent scheme, thereby allowing us to investigate the properties of the continuous trajectories and, in turn, gain a deeper understanding of the convergence properties of the corresponding iterative optimization algorithms.  The connecting bridge between continuous and discrete dynamics is provided by stochastic approximation theory \cite{benaim2005stochastic,borkar2009stochastic}.  Moreover, since our differential inclusion is formulated on a general smooth manifold with boundary, it naturally induces new iterative algorithms for nonconvex, nonsmooth optimization over \(\cM\cap\overline C\).  Before deriving these algorithms, we first establish some basic properties of the flow \eqref{eq:DI_manifold}.

\subsection{Basic properties of interior mirror descent flow}
We begin by showing that the differential inclusion \eqref{Eq_def_DI} is well-posed, i.e.\ there exists a global solution, such that $\{x(t):t\in[0,\infty)\}$ existing.
This result is a natural nonsmooth extension of Theorem 4.1 in \cite{alvarez2004hessian}. Indeed, by \cite{aubin2012differential}, the trajectory of \eqref{Eq_def_DI} exists whenever $\cH$ is a locally bounded graph-closed set-valued mapping. To proceed, we introduce the following assumption to ensure the well-posedness. 
\begin{assumption}
\label{assumption:well_posed}
    \begin{enumerate}
        \item $\phi$ is Legendre function over $C$, $\phi\in\cC^2(C)$, and $\nabla^2\phi(x)$ is locally Lipschitz continuous. $f$ is path-differentiable.
        \item One of the following holds:
        \begin{enumerate}
            \item $\{x\in\cM\cap C:f(x)\leq f(x_0)\}$ is bounded.
            \item $\norm{\nabla^2\phi(x)^{-1}}\leq K\norm{x}+M$, for some $K,M>0$.
        \end{enumerate}
    \end{enumerate}
\end{assumption}

It is worth noting that Assumption \ref{assumption:well_posed}.2 can be replaced by the condition that $x(t)$ is uniformly bounded if $\{x(t):t\in[0,T]\}$ exists for some $T>0$. Assumption \ref{assumption:well_posed}.2 provides a sufficient condition to guarantee the boundedness of $x(t)$. For example, if the manifold $\mathcal{M} \cap C$ is itself bounded, this condition is trivially satisfied. With these assumptions in place, we state the following proposition. Since the proof follows the same lines as in \cite[Theorem 4.1]{alvarez2004hessian}, we omit it here. 
\begin{proposition}
\label{prop:well_pose}
Suppose Assumption \ref{assumption:well_posed} holds. Then, the differential inclusion \eqref{eq:DI_manifold} is well-posed for all $t \in (0, \infty)$, and the trajectory satisfies $x(t) \in C \cap \mathcal{M}$. Furthermore, the function $f(x(t))$ is non-increasing and converges as $t \to \infty$.
\end{proposition}
This proposition ensures that the trajectory remains within the relative interior of the feasible set $C \cap \mathcal{M}$. Consider the case where $\mathcal{M}$ is an affine space. Here, the tangent space $T_x \mathcal{M} = \mathcal{M}$, and the projection $\mathrm{P}_{T_x \mathcal{M}}$ simplifies to $\mathrm{P}_x$, as defined in Section 2.3. A natural discrete approximation of the differential inclusion \eqref{eq:DI_manifold} is given by:
\[
x_{k+1}\in x_k-\eta_k {\rm P}_x\nabla^2\phi(x)^{-1}\partial f(x_k),
\]
where $\eta_k > 0$ is a step size. This scheme exactly corresponds to the Hessian-barrier method proposed by \cite{bomze2019hessian}, when $f$ is differentiable. However, the connection to the mirror descent scheme is not immediately obvious from this form. To make this connection explicit, an equivalent reformulation of \eqref{eq:DI_manifold} can be derived.
\begin{proposition}
\label{prop:DI_equiv}
Equations \eqref{eq:DI_manifold} is equivalent to the following differential inclusion
\begin{equation}
\label{eq:DI_manifold_Breg}
\frac{d}{dt}\nabla\phi(x(t))\in-\partial f(x(t))+N_{x(t)}\cM,\;x(0)\in C\cap\cM,
\end{equation}
where $ N_x \mathcal{M} = (T_x \mathcal{M})^\perp $ denotes the normal space to the tangent space at $ x $.
\end{proposition}
\begin{proof}
Consider any $x(t)$ satisfying \eqref{eq:DI_manifold}, there exist $d(x)\in\partial f(x)$, such that $\dot x=-{\rm P}_{T_x\cM}\nabla^2\phi(x)^{-1}d(x)$. By definition of ${\rm P}_{T_x\cM}$, $\dot x$ is the solution of the following convex optimization problem: 
\begin{equation}
\begin{aligned}
\min\;&\frac{1}{2}\norm{u+\nabla^2\phi(x)^{-1}d(x)}_x^2\\
\text{s.t. }&u\in T_x\cM.
\end{aligned}
\label{eq:equiv-prob}
\end{equation}
By optimality condition, we have 
\[
0\in\nabla^2\phi(x)(\dot x+\nabla^2\phi(x)^{-1}d(x))+N_x\cM.
\]
Thus $\frac{d}{dt}\nabla\phi(x)\in -d(x)+N_x\cM$, and satisfying \eqref{eq:DI_manifold_Breg}. 

Now, for any $x(t)$ satisfying \eqref{eq:DI_manifold_Breg}. There exists $d(x)\in\partial f(x)$, and $\nabla^2\phi(x)\dot x\in-d(x)+N_x\cM$ , which is equivalent to $\dot x$ is the solution of the optimization problem \eqref{eq:equiv-prob}. Since \eqref{eq:equiv-prob} is a strongly convex problem, it has a unique solution. By definition of ${\rm P}_{T_x\cM}$, we have $\dot x\in-{\rm P}_{T_x\cM}\nabla^2\phi(x)^{-1}d(x)$. Thus, $x(t)$ satisfies \eqref{eq:DI_manifold}. This completes the proof.
\end{proof}
Now, we can explicitly see how \eqref{eq:DI_manifold_Breg} leads to mirror descent scheme. When $\cM=L$ is affine space, discretization of \eqref{eq:DI_manifold_Breg} gives 
\[
\frac{\nabla\phi(x_{k+1})-\nabla\phi(x_k)}{\eta_k}\in-d_k-N_{x_k}\cM,
\]
where $d_k\in\partial f(x_k)$. This is the optimality condition for the convex problem:
\[
\begin{aligned}
    \min_{x\in\R^n}\;&\inner{d_k,x}+\frac{1}{\eta_k}\cD_\phi(x,x_k)\\
    \text{s.t. }&x_{k+1}\in x_k+T_{x_k}\cM=L.
\end{aligned}
\]
Since $\phi$ is strict convex, the problem has a unique solution. Thus, this naive discretization corresponds to the conventional mirror descent scheme when $ f $ is differentiable. As the consequence of Proposition \ref{prop:DI_equiv}, we have the following equivalent description of the stable set. The proof is similar to the proof of Proposition \ref{prop:equiv_stable}, thus omitted.
\begin{proposition}
\label{prop:stable_equiv}
$\{x:0\in H(x)^{-1}(\partial f(x)+N_{x}\cM)\}=\{x:0\in {\rm P}_{T_x\cM}H(x)^{-1}\partial f(x)\}$.
\end{proposition}

The stable set is a core concept in the dynamical system approach to stochastic approximation, typically linked to the stationary points of an optimization problem. For the differential inclusion \eqref{eq:DI_manifold}, the stable set is defined as:
\begin{equation}
\cS=\{x:0\in{\rm P}_{T_x\cM}\nabla^2\phi(x)^{-1}\partial f(x)\}.
\label{def:stable_set}
\end{equation}
In comparison, the stationary point set of the optimization problem is given by:
\begin{equation}
\Omega=\{x:0\in\partial f(x)+N_{\cM}(x)+N_{\overline C}(x)\}.
\label{def:stationary_set}
\end{equation}
Interestingly, the stable set of the continuous dynamical system (1) can be strictly larger than the stationary point set of the optimization problem. This observation ties into the notion of spurious stationary points in mirror descent schemes, a phenomenon recently highlighted in the optimization community \cite{chen2024spurious,zhang2024stochastic}. To illustrate this, we present two examples:
\begin{example}
\label{example:stable_set}
\begin{enumerate}
    \item For $C=\{x:x\geq0\}$, $\cM=\R^n$, $\phi(x)=\sum_{i=1}^nx_i(\log x_i-1)$. The stable set of the associated differential inclusion is:
    \[
    \cS=\qty{x\in\R^n: 0\in x\circ\partial f(x)}.
    \]
    Here, $ x = 0 $ is always in $ \mathcal{S} $, even if it is not a stationary point of the optimization problem.
    \item Consider $C=\{x:\norm{x}\leq1\}$, $\cM=\R^n$, $\phi(x)=-\sqrt{1-\norm{x}^2}$. Simple calculation shows $\nabla^2 \phi(x)^{-1} = \sqrt{1 - \|x\|^2} (I - x x^T).$
    Thus, the stable set is:
    \[
    \cS=\qty{x\in\R^n: 0\in \sqrt{1 - \|x\|^2} (I - x x^T)\partial f(x)}.
    \]
    In this case, any $ x $ on the boundary of $ C $ is in the stable set.
\end{enumerate}
\end{example}

In the following, we demonstrate that under specific conditions on the kernel function, which are satisfied by many widely used barrier functions, the stationary points of an optimization problem are always contained within the stable set of the associated differential inclusion. Additionally, we provide a clearer characterization of the stable set.
\begin{proposition}
\label{prop:relation_NC_Null}
Suppose that for any $\rho>0$, $\sup_{x\in C\cap\rho\mathbb{B}}\left\{\norm{\nabla^2\phi(x)^{-1}\nabla\phi(x)}\right\}<\infty.$ Then, it holds that 
\[
N_{\overline{C}}(x)\subset{\rm Null}((\nabla^2\phi(x))^{-1}).
\]
In addition, if ${\rm span}(\partial^\infty\phi(x))\supset{\rm Null}((\nabla^2\phi(x))^{-1})$, the stable set of differential inclusion can be characterized as 
\[
\cS=\{x:0\in\nabla f(x)+N_\cM(x)+{\rm span}(N_{\overline{C}}(x))\}.
\]
\end{proposition}
\begin{proof}
Define $S:=\left\{d:\nabla^2\phi(x_k)e_k\rightarrow d,e_k\rightarrow0,x_k\rightarrow\bar x\right\}$. For any $v\in{\rm span}(\partial^\infty\phi(x))$, there exists $\lambda_k\rightarrow0$, $x_k\rightarrow\bar x$, and $\lambda_k\nabla\phi(x_k)=:v_k\rightarrow v$. Let $e_k:=\lambda_k(\nabla^2\phi(x_k))^{-1}\nabla\phi(x_k)$, by assumption, $e_k\rightarrow0$. Thus, $\nabla^2\phi(x_k)e_k=\lambda_k\nabla\phi(x_k)\rightarrow v\in S$. Note that $\phi$ is Legendre of $\overline{C}$, thus ${\rm span}(\partial^\infty\phi(x))=N_{\overline{C}}(x)$. Furthermore, if ${\rm span}(\partial^\infty\phi(x))\supset{\rm Null}((\nabla^2\phi(x))^{-1})$, we have ${\rm span}(\partial^\infty\phi(x))={\rm Null}((\nabla^2\phi(x))^{-1})$. This completes the proof.
\end{proof}
\begin{example}
The two conditions in the proposition are satisfied by several commonly used kernel functions, including:
    \begin{enumerate}
        \item The kernel function $\phi(x)=\sum_{i=1}^n(x_i(\log x_i-1))$, or $\phi(x)=-\sum_{i=1}^n\log x_i$, or $\phi(x)=-\sum_{i=1}^n(x_i)^{2-p}$ where $(p\in(1,2))$, for nonnegative orthant $C=\{x\in\R^n:x\geq0\}$.
        \item The kernel function $\phi(X)=-\log(\det(X))$ for positive semidefinite (PSD) cone.
    \end{enumerate}
\end{example}

By Proposition \ref{prop:relation_NC_Null} and Example \ref{example:stable_set}, we have the following:
\begin{proposition}
Suppose conditions in Proposition \ref{prop:relation_NC_Null} holds. Then, $\Omega\subset\cS$. Moreover, there exists a problem \eqref{prob:PCP} and a kernel $\phi$, such that the inclusion holds strictly.   
\end{proposition}

\begin{remark}
Classical dynamical systems theory shows that the $ \omega $-limit set of any trajectory of the differential inclusion is contained in $ \mathcal{S} $. As a result, spurious stationary points may emerge in the trajectories. In subsequent sections, we will illustrate that certain spurious stationary points observed in mirror descent schemes correspond exactly to points in $ \mathcal{S} \setminus \Omega $.
\end{remark}

\subsubsection{Connection with spurious stationary point of mirror descent scheme in \cite{chen2024spurious}}
In this part, we examine the relationship between the stable set of the differential inclusion defined in equation \eqref{eq:DI_manifold_Breg} and the spurious stationary points of the mirror descent scheme, as explored in Chen et al. (2024). Specifically, when $\cM=\{x:Ax=b\}$ and $C=\R^n_+$ (special box constraint), \cite{chen2024spurious} demonstrate that the mirror descent scheme can exhibit spurious stationary points. These points are fixed points of an extended Bregman gradient mapping but do not correspond to the true stationary points of the original optimization problem. We show that these fixed points is exactly the stable set of the differential inclusion \eqref{eq:DI_manifold}.

We begin by introducing the extended update mapping for the mirror-descent scheme from \cite{chen2024spurious}, generalize it to the nonsmooth setting, and restate it in our notation. The kernel function employed here is separable Legendre function over $\R^n_{+}$, i.e. 
\begin{assumption}
\label{assumption:separable_phi}
The kernel function is separable $\phi(x)=\sum_{i=1}^d\phi_i(x_i):C\rightarrow\R$ satisfying the following conditions:
\begin{enumerate}
    \item Each $\phi_i$ is a Legendre function over $\R_+$, i.e. $\phi_i$ is differentiable strict convex over $\R_+$, and $\lim_{x_i\rightarrow0_+}\phi'(x_i)=+\infty$.
    \item There exist $\gamma>0$ and continuous $\psi_i:\R\rightarrow\R_{++}$, such that $\phi''_i(x_i)=x_i^{-\gamma}\psi_i(x_i)$.
\end{enumerate}
\end{assumption}
Examples satisfying Assumption \ref{assumption:separable_phi} include 
\[
\phi_i(t)=\left\{
\begin{array}{cc}
    t\log t, & p=1 \\
    \frac{1}{(2-p)(1-p)}t^{2-p}, &\;\; p\in(1,2) \\
    -\log t, & p=2.
\end{array}\right.
\]

To address points on the boundary of the feasible set, \cite{chen2024spurious} defines an extended update mapping for the mirror descent scheme.

\begin{definition} 
Define the extended mapping \( \overline{T}_\eta(x) \) as 
\[
\overline{T}_\eta(x)\in\argmin\qty{\inner{\partial f(x),y} + \frac{1}{\eta} \sum_{i \in \mathcal{I}(x)} D_\phi(y_i, x_i) + \delta_{{y}_{\mathcal{J}(x)} = x_{\mathcal{J}(x)}}(y):\;\text{s.t. } Ax=b,\;x\geq0},
\]
where $\cI(x)=\{i:x_i>0\}$ and $\cJ=\{j:x_j=0\}$. A point \(x \in \mathcal{X}\) is a spurious stationary point if and only if
\[
x\in\overline{T}_\eta(x) \, \text{ but } \, 0 \notin \partial f(x)+N_x\cM+N_{\overline{C}}(x).
\]
\end{definition}
We now prove that the fixed points of $ \overline{T}_\eta $ are equivalent to the stable set $ \mathcal{S} $ of the differential inclusion \eqref{eq:DI_manifold}. Consequently, the spurious stationary points of the mirror descent scheme are precisely those points in $ \mathcal{S} $ that are not true stationary points of the original problem.
\begin{proposition}
The fixed point set of $\overline{T}_\eta$ is equivalent to the stable set $\cS$ of the corresponding stable set for differential inclusion \eqref{eq:DI_manifold}.
\end{proposition}
\begin{proof}
Note that the normal cone  $N_{\overline{C}\cap\cM}(x)=\qty{A^T\mu-\lambda:\mu\in\R^m,\lambda_{\cI(x)}=0,\lambda_{\cJ(x)}\geq0}$. Moreover, 
\[
\partial\delta_{{y}_{\mathcal{J}(x)} = x_{\mathcal{J}(x)}}(y)\bigg|_{y=x}=\qty{\xi\in\R^n:\xi_{\cI(x)}=0}.
\]
$x\in\overline{T}_\eta(x)$ is equivalent to the existence of $d(x)\in\partial f(x)$ such that 
\[
\begin{aligned}
0=&\;d(x)+0+\partial\delta_{{y}_{\mathcal{J}(x)} = x_{\mathcal{J}(x)}}(y)\bigg|_{y=x}+N_{\overline{C}\cap\cM}(x)\\
=&\;d(x)+\{A^T\mu:\mu\in\R^m\}+\{\lambda\in\R^n:\lambda_{\cI(x)}=0\}\\
\in&\;\partial f(x)+N_x\cM+{\rm span}\qty(N_{\overline{C}}(x)).
\end{aligned}
\]
By Proposition \ref{prop:relation_NC_Null}, this condition is equivalent to the stable set of \eqref{eq:DI_manifold}. 
\end{proof}

We now make some remarks on mirror descent for the unconstrained optimization problem, i.e. $C=\cM=\R^n$. In this case, the "infinity point" (i.e., points where \(\|x\| \to \infty)\) can be regarded as the boundary of \( C \). If \(\phi\) is a Legendre function over \(\mathbb{R}^n\), the infinity point may lie in the stable set of the differential inclusion (1), which aligns with the counterexample constructed in \cite{zhang2024stochastic}. Specifically, their counterexample shows that the commonly used Bregman stationary measure approaches zero as \( x \) tends to infinity.

For example, consider the case where $C=\cM=\R^n$ and kernel function $\phi(x)=\frac{1}{2}\norm{x}^2+\frac{1}{4}\norm{x}^4$.
The stable set is 
\[
\begin{aligned}
\cS=&\qty{x\in\R^n:0\in\nabla^2\phi(x)^{-1}\partial f(x)}\\
=&\qty{x\in\R^n:0\in\qty(\frac{1}{1 + \|x\|^2} I - \frac{2}{(1 + \|x\|^2)(1 + 3 \|x\|^2)} x x^T)\partial f(x)}.
\end{aligned}
\]
Suppose $\liminf_{\norm{x}\to\infty}\dist\qty(0,\partial f(x))<\infty$, then we have
\[
0\in\limsup_{x:\norm{x}\to\infty}\qty{\qty(\frac{1}{1 + \|x\|^2} I - \frac{2}{(1 + \|x\|^2)(1 + 3 \|x\|^2)} x x^T)\partial f(x)}
\]
Thus, the infinity point can be considered a stable point in the stable set \(\mathcal{S}\), illustrating how mirror descent may converge to non-stationary points even for unconstrained problem.

\subsubsection{Connection with reparameterization}
As shown in \cite{li2022implicit}, the mirror descent flow is closely related to the gradient flow of a reparameterized model. Consider the optimization problem where $ \mathcal{M} = \mathbb{R}^n $:
\[
\min\;f(x),\quad\text{ s.t. } x\in\overline{C},
\]
with the mirror descent flow defined as:
\[
\dot x\in-\nabla^2\phi(x)^{-1}\partial f(x).
\]
Suppose the constraint set $ \overline{C} $ can be reparameterized as $ \{ G(y) \in \mathbb{R}^n : y \in \mathbb{R}^m \} $, where $ G $ is a differentiable function. By setting $ x = G(y) $, the optimization problem becomes:
\[
\min_{y\in\R^m}\;f(G(y)).
\]
Assuming the chain rule applies to $ f \circ G $, the subgradient flow for the reparameterized problem is:
\[
\dot y\in-\nabla G(y)\partial f(G(y)).
\]
Since $ \dot{x} = \nabla G(y) \dot{y} $, we derive:
\[
\dot x\in\nabla G(y)^T\dot y\subset-\nabla G(y)^T\nabla G(y)\partial f(G(y))=-\nabla G(y)^T\dot y\subset-\nabla G(y)^T\nabla G(y)\partial f(x).
\] 
Thus, if the condition:
\begin{equation}
\nabla^2\phi(x)^{-1}=\nabla G(y)^T\nabla G(y), \text{ with $x=G(y)$,}
\label{eq:equiv_repara_mirror}
\end{equation}
holds, the mirror descent flow is equivalent to the subgradient flow of the reparameterized model.

\begin{example}[Squared Slack Variable]
Consider $ \phi(x) = \sum_{i=1}^n x_i (\log x_i - 1) $ and the constraint set $ C = \{ x \in \mathbb{R}^n : x \geq 0 \} $. The reparameterization $ x = \frac{1}{2} y \circ y $ (where $ \circ $ denotes element-wise multiplication) satisfies the above condition. Known as the squared slack variable, this approach has recently gained attention for addressing inequality-constrained problems, as explored in \cite{kolb2023smoothing,ding2023squared,tang2024optimization}.  
\end{example}

Recent research, such as \cite{levin2025effect,tang2024optimization}, has investigated reparameterized methods in Riemannian optimization, uncovering relationships between the reparameterized and original problems. Under certain regularity conditions, a second-order stationary point $ y $ of the reparameterized problem implies a first-order stationary point $ x = G(y) $ of the original problem (denoted $ 2 \Rightarrow 1 $). However, a first-order stationary point of the reparameterized problem does not necessarily imply a first-order stationary point of the original problem (denoted $ 1 \not\Rightarrow 1 $). When the condition $ \nabla^2 \phi(x)^{-1} = \nabla G(y) \nabla G(y)^T $ holds, the mirror descent flow aligns with the gradient flow of the reparameterized problem. In this case, the $ 1 \not\Rightarrow 1 $ phenomenon can be attributed to spurious stationary points in the mirror descent flow.

\begin{example}
Consider the constrained optimization problem:
$$\min_{x \in \mathbb{R}^n} f(x), \quad \text{s.t.} \quad x \geq 0,$$ 
and its reparameterized form with $ x = y \circ y $: 
$$\min_{y \in \mathbb{R}^n} f(y \circ y).$$
The first-order stationary points of the reparameterized problem are: $\{ y : 0 \in y \circ \partial f(y \circ y) \},$ which, since $ x = y \circ y \geq 0 $, is equivalent to: $0 \in \sqrt{x} \circ \partial f(x).$ For the kernel function $ \phi(x) = \sum_{i=1}^n (x_i \log x_i - x_i) $, the stable set of the mirror descent flow is: $\{ x : 0 \in \nabla^2 \phi(x)^{-1} \partial f(x) \}$, which is exactly the same as the first-order stationary points of the reparameterized problem. The $ 1 \not\Rightarrow 1 $ phenomenon arises due to spurious stationary points in the mirror descent flow.
\end{example}


\subsection{Asymptotic convergence of the trajectory}
In this section, we explore the asymptotic convergence properties of the trajectory $ x(t) $, with a focus on its long-term behavior near points in $ \mathcal{S} $ that are not true stationary points of the optimization problem (spurious stationary points). We first demonstrate that any limit point of $ x(t) $ lies in the stable set $ \mathcal{S} $ of the differential inclusion \eqref{eq:DI_manifold}. Subsequently, we analyze the trajectory's behavior near the spurious points.

We begin by establishing that the trajectory $ x(t) $ exhibits subsequential convergence to the stable set $ \mathcal{S} $. Unlike prior works such as \cite{duchi2018stochastic,davis2020stochastic,bolte2021conservative}, our approach directly study the continuous trajectory.
\begin{lemma}
Suppose $\phi$ is a $\cC^2({C})$ Legendre function, and $f$ is path-differentiable. Then, any limit point of $x(\cdot)$ lies in the stable set $\cS$.     
\end{lemma}
\begin{proof}
First, we note that for any trajectory $x(t)$, $f(x(t))$ is nonincreasing:
\[
\frac{d}{dt}f(x(t))=-\inner{\partial f(x(t)),\dot x(t)}=-\norm{\dot x(t)}^2_{\nabla^2\phi(x)}\leq0.
\]
For any trajectory $x(t)$ with $x(0)\notin\cS$, there exists a $T>0$ such that
\[
f(x(T))<f(x(0)).
\]
Otherwise, we have $\norm{\dot x(t)}^2_{\nabla^2\phi(x)}=0$ for almost all $t>0$. Thus, $x(t)\equiv x(0)$, and $\dot x=0\in{\rm P}_{x(0)}(H(x(0))^{-1}\partial f(x(0)))$. This is contradictory to that $x(0)\notin\cS$. Thus $f$ is a Lyapunov function of \eqref{eq:DI_manifold}.

Assume $x(t_j)\rightarrow x^*$. Let $x^{t_j}(\cdot):=x(t_j+\cdot)$. By uniformly Lipschitz continuity, $\{x^{t_j}(\cdot)\}$ is relatively compact over $\cC(\R^n_+)$. Suppose $x^{t_{j_k}}(\cdot)\rightarrow\bar x(\cdot)$ in $\cC(\R^n)$. Note that for any $T>0$:
\[
f(\bar x(T))=\lim_{k\rightarrow\infty}f(x^{t_{j_k}}(T))=\lim_{k\rightarrow\infty}f(x(t_{j_k}+T))=\lim_{t\rightarrow\infty}f(x(t))=f(x^*).
\]
On the other hand, $\bar x$ is also the trajectory of \eqref{eq:RGF2} with $\bar x(0)=x^*$. If $x^*\notin\cS$, since $f$ is a Lyapunov function, there exists $T>0$ such that
\[
f(\bar x(T))<f(\bar x(0))=f(x^*).
\]
This leads to a contradiction. Thus, $x^*\in\cS$. This complete the proof.
\end{proof}

We now investigate the trajectory's behavior near the spurious stationary points. We start with the case of linear constraints and later extend the analysis to nonlinear constraints. Assume the conditions of Proposition \ref{prop:relation_NC_Null} hold, the stable set is:
\[
\cS=\{x:0\in[\nabla^2\phi(x)]^{-1}(\partial f(x)+N_L)\}=\{x:0\in\partial f(x)+N_L+{\rm span}(N_{\overline C}(x))\}.
\]
Note that here we have already assume $\nabla^2\phi(x)^{-1}$ is extendable to $\cM\cap\overline{C}$.  Apparently, this stable set is larger than the stationary point of the optimization problem:
\[
\Omega=\{x:0\in\partial f(x)+N_L+N_{\overline C}(x)\}.
\]
By choices of $\phi$, $H(x)^{-1}={\rm Diag}(x_i^\gamma\psi_i(x_i)^{-1})_{i\in[d]}$, thus, the stable set is equivalent to 
\[
\cS=\{x:0\in(x\circ(\partial f(x)+N_L))\}=\{x:0\in {\rm P}_x(x\circ\partial f(x)\}.
\]
Here, ${\rm P}_x$ is the projection over $L$ under the metric induced by the Hessian of the entropy function $\sum_{i=1}^n(x_i\log x_i-x_i)$, whose Hessian inverse is ${\rm Diag}(x)$. Thus, in this case, we can fix the kernel function as $\sum_{i=1}^n(x_i\log x_i-x_i)$.

Define the index sets $ \mathcal{I}(x) = \{ i : x_i > 0 \} $ and $ \mathcal{J}(x) = \{ j : x_j = 0 \} $ (with the dependence on $ x $ omitted for brevity where clear). For any $ x \in \mathcal{S} $, there exists $ d(x) \in \partial f(x) $ such that:
\[
\left\{\begin{array}{cc}
    x_{\cI}>0, &d(x)_\cI=A^T_{\cI}y, \text{ for some }y\in\R^m,\\
    x_{\cJ}=0, & \text{ any }d(x)_{\cJ}.
\end{array}\right.
\]
For any $x\in\Omega$, there exists $d(x)\in\partial f(x)$, such that 
\[
\left\{\begin{array}{cc}
    x_{\cI}>0, &d(x)_{\cI}=A^T_{\cI}y, \text{ for some }y\in\R^m,  \\
    x_{\cJ}=0, & d(x)_{\cJ}\geq A^T_{\cJ}y, \text{ for some }y\in\R^m.
\end{array}\right.
\]
The following lemma provides an alternative characterization of $ \mathcal{S} \setminus \Omega $:
\begin{lemma}
\label{le:alternative}
Suppose Assumption \ref{assumption:separable_phi} holds, for any $x\in\cS\setminus\Omega$, we have
\[
\left\{\begin{array}{cc}
    x_{\cI}>0, &\partial_{\cI}f(x)\ni A^T_{\cI}y_{\cI}, \text{ for some }y_{\cI}\in\R^{|\cI|}  \\
    A_{\cJ}u=0,\;u\geq0 & \inner{u,\partial_{\cJ}f(x)}<0\text{ for some }u\in\R^{|\cJ|} .
\end{array}\right.
\]
\end{lemma}
\begin{proof}
Define $S:=\{s:s\geq A^T_{\cJ}y_{\cJ}, \text{ for some }y_\cJ\}$. Since $x\in\cS\setminus\Omega$, $S$ is a convex set and $\partial_\cJ f(x)\cap S=\emptyset$. By separate theorem, we have there exists $u\in\R^{|\cJ|}$, such that
\[
\inner{u,d(x)_\cJ}<\inner{u,r}+\inner{u,A^T_\cJ y_\cJ}, \text{ for any }r\geq0,\;y_\cJ\in\R^{|\cJ|},\;d(x)\in\partial f(x).
\]
Since $y_\cJ$ is arbitrary, we have $A_\cJ u=0$. Since $r$ is any nonnegative vector, we have $u\geq0$. Finally, we get $\inner{u,d(x)_\cJ }<0$, for any $d(x)\in\partial f(x)$.
\end{proof}

We now demonstrate that the trajectory $ x(t) $ can exit a neighborhood of a spurious stationary point in finite time. However, this finite‐time exiting result does not rule out the possibility that the trajectory may later re‐enter the neighborhood.

\begin{proposition}
\label{prop:exiting}
Suppose Assumption \ref{assumption:separable_phi} holds. For any spurious stationary point $ x \in \mathcal{S} \setminus \Omega $, there exists a neighborhood $ \mathcal{N} $ such that if the trajectory $ x(t) $ enters $ \mathcal{N} $ at time $ \tau_k $, defined as:
\[\tau_k:=\inf\qty{t\geq T_{k-1}:x(t)\in\cN\cap\cM\cap C},
\]
it exits $ \mathcal{N} $ in finite time $ T_k $, where:
\[
T_k:=\sup\qty{T:x(t)\in\cN\cap\cM\cap C:\text{ for any $t\in[\tau_k,\tau_k+T)$}},
\]
and $T_{-1}:=0$.
\end{proposition}

\begin{proof}


Assume $ x(t) $ has a limit point $ \bar{x} \in \mathcal{S} \setminus \Omega $, with $ \bar{x}_{\mathcal{I}} > 0 $ and $ \bar{x}_{\mathcal{J}} = 0 $. By the lemma, there exists $ u_{\mathcal{J}} \geq 0 $ such that:
$$\langle u_{\mathcal{J}}, \nabla^2 \phi(x)_{\mathcal{J}} \dot{x}_{\mathcal{J}} \rangle = -\langle u_{\mathcal{J}}, \partial_{\mathcal{J}} f(x) \rangle > \delta > 0,$$
for some $ \delta > 0 $ and for $ x $ sufficiently close to $ \bar{x} $. Without loss of generality, set $ \bar{x} = 0 $ (or restrict attention to $ \mathcal{J}(\bar{x}) $). Consider a neighborhood $ \mathcal{N} = (0, \epsilon)^n \cap \mathcal{M} \cap C $, where:
$$\sum_{i=1}^n \frac{u_i}{x_i} \dot{x}_i \geq \delta.$$
Thus, there exists an index $ \bar{i} $ and $ \delta_0 > 0 $ such that $ \frac{u_{\bar{i}}}{x_{\bar{i}}} \dot{x}_{\bar{i}} \geq \delta_0 $. For a fixed $ k $ (omitting the subscript for simplicity), suppose by contradiction that $ x(t) \in \mathcal{N} $ for all $ t \geq \tau $. Since $ x(\tau) $ is in the relative interior, $ x_{\bar{i}}(\tau) \in (0, \epsilon) $. Applying Gronwall’s inequality:
$$x_{\bar{i}}(t) \geq x_{\bar{i}}(\tau) \exp \left( \frac{\delta_0}{u_{\bar{i}}} (t - \tau) \right),$$
which implies $ x_{\bar{i}}(t) \to \infty $ as $ t \to \infty $, contradicting the boundedness of $ \mathcal{N} $. Hence, $ x(t) $ exits $ \mathcal{N} $ in finite time $ T $.
\end{proof}

As a direct consequence of Proposition \ref{prop:exiting}, if the trajectory $x(t)$ of \eqref{eq:DI_manifold} converges, it must converge to a stationary point of the original optimization problem. Indeed, convergence to a spurious stationary point would contradict Proposition \ref{prop:exiting}, which ensures that the trajectory exit the neighborhood of such points in finite time.

\begin{proposition}
    \label{prop:covergent_traj}
    Suppose the trajectory $x(t)$ of \eqref{eq:DI_manifold} converges to a point $x^*$. Then, $x^*\in\Omega$. 
\end{proposition}

Note that \cite[Theorem 4.2]{alvarez2004hessian} proves that when $f$ is quasi-convex and $\phi$ is Bregman function, $x(t)$ converges, thus as a byproduct, we prove that $x(t)$ converge to the stationary point. As a remark, their approach is quite different from ours, where they directly prove the convergence of this special function class, while we prove for a more general case for the finite time escaping property, and as the byproduct of our general results, we prove the convergence to stationary point. 
However, this result does not imply $x(t)$ never subsequentially converge to the spurious stationary point. In the next section, we will explore the subsequential avoidance of  spurious stationary point.

In \cite[Theorem 4.2]{alvarez2004hessian}, it is shown that for quasi-convex $f$ and a Bregman function $\phi$, the trajectory $x(t)$ converges. Thus as a byproduct, we prove that $x(t)$ converge to the stationary point in their setting. Our approach differs: we establish a more general finite-time exit property from spurious stationary points.
However, this result does not preclude the possibility that subsequences of $x(t)$ might approach spurious stationary points. In the next section, we will investigate conditions under which subsequential convergence to spurious stationary points is avoided.

We now extend these results to the more general case of \textbf{nonlinear} constraints, where the feasible set is defined as $C = \{ x : g_i(x) < 0, \, i = 1, \dots, m \}$ with smooth functions $g_i$. For points $x \in \mathcal{S} \setminus \Omega$, we show that there exists a direction along which the trajectory is repelled, ensuring escape from neighborhoods of spurious stationary points.

\begin{lemma}
For any \( x \in \cS \setminus \Omega \), there exists \( u \in T_x \cM \) such that:
\[
\begin{array}{c}
\langle u, \nabla g_i(x) \rangle \leq 0, \quad \langle u, \partial f(x) \rangle < 0,\quad \forall i \in \cJ(x).
\end{array}
\]
\end{lemma}

\begin{proof}
Note that $\cS=\qty{x:0\in\partial f(x)+\nabla c(x)\mu+\nabla g_{\cJ}(x)\lambda, g_\cJ(x)=0}$, and

$\Omega=\qty{x:0\in\partial f(x)+\nabla c(x)\mu+\nabla g_{\cJ}(x)\lambda, \lambda\geq0, g_\cJ(x)=0}$. Define $S=\qty{\nabla c(x)\mu+\nabla g_{\cJ}\lambda:\lambda\geq0}$. For any $x\in\cS\setminus\Omega$, we have $-\partial f(x)\cap S=\emptyset$. By separate theorem, there exists $u\in\R^{n}$, such that
\[
\inner{-\partial f(x),u}>\sup_{\mu,\lambda\geq0}\inner{u,\nabla c(x)\mu+\nabla g_{\cJ}(x)\lambda}.
\]
Since $\mu$ is arbitrary, we have $\nabla c(x)^Tu=0$. Thus, $u\in T_x\cM$. Moreover, $\lambda$ is an arbitrary positive vector, $\nabla g_\cJ(x)^Tu\leq0$. This completes the proof.
\end{proof}
Similar to the proof of Proposition \ref{prop:exiting}, there exists a neighborhood of $x\in\cS\setminus\Omega$, such that any $x(t)$ in this neighborhood,
\[
\inner{u,\nabla^2\phi(x)\dot{x}}>\delta>0.
\]
Analogous to the linear case, for $x \in \mathcal{S} \setminus \Omega$, there exists a neighborhood where the trajectory $x(t)$ satisfies:
$$\langle u, \nabla^2 \phi(x) \dot{x} \rangle > \delta > 0,$$
for some $u \in T_x \mathcal{M}$ as identified in the lemma. This condition ensures that the trajectory exits the neighborhood in finite time, mirroring the results of Proposition \ref{prop:exiting} and Proposition \ref{prop:covergent_traj} for the nonlinear constraint setting.

\section{Spurious stationary points avoidance}
The classical results for differential inclusions (DI) indicate that the $\omega$-limit set of the trajectory generated by DI \eqref{eq:DI_manifold} is contained within the stable set $\mathcal{S}$. In this section, we demonstrate that under certain constraint qualifications (CQ) and isolated conditions for $\mathcal{S} \cap \partial C$, the trajectory converges or subsequentially converges to a true stationary point in $\Omega$. When these regularity conditions are not satisfied, we propose a novel perturbation strategy to ensure that the trajectory converges to true stationary point of the perturbed problem.

\subsection{Avoidance under complementarity condition}
In this subsection, we establish that spurious stationary points can be avoided under a complementarity condition. We begin with the linear constraint case, where $C = \{x \in \mathbb{R}^n : x > 0\}$ and $\mathcal{M} = \{x \in \mathbb{R}^n : Ax = b\}$, and later extend the results to nonlinear constraints.

Recall that the moving direction $v(x)=-{\rm P}_xH(x)^{-1}d(x)$, and dual variable and dual slack variable are given by
\[
\begin{aligned}
   v(x)=&-{\rm P}_xH(x)^{-1}d(x), \\
   y(x)=&(AH(x)^{-1}A^T)^{-1}AH(x)^{-1}d(x).
\end{aligned}
\]
Next, we show that the trajectory will never converge to the spurious stationary point, in which the complementarity condition is strictly violated. The intuition is that such points possess a neighborhood from which the trajectory is repelled from any directions, preventing subsequential convergence. Furthermore, if the complementarity condition holds for the $\omega$-limit set, the trajectory converges to a true stationary point.

\begin{theorem}
\label{thm:completmentarity}
Let $x(t)$ be the trajectory of \eqref{eq:DI_manifold}. It holds that
\begin{enumerate}
    \item[(i)] For any point $x$ in the stable set, such that $s(x)+x<0$ holds, then, $x(t)$ will never subsequentially converge to $x$.
    \item[(ii)] If the complementarity condition: $s(x)+x\geq0$ holds at $\omega(x_0)$, then any limit point of the trajectory $\{x(t)\}$ of \eqref{eq:DI_manifold} is the stationary point of the original problem \eqref{prob:PCP}.
\end{enumerate}
\end{theorem}

\begin{proof}
For points in the relative interior of the feasible set, the result holds trivially, so we focus on $\mathcal{S} \cap \partial C$. We prove (ii) first. Since $\omega$-limit set is included in the stable set $\{x:0\in {\rm P}_x(x\circ\partial f(x)\}$. By definition of ${\rm P}_x$, we have 
\[
d(x)+H(x)v(x)-A^Ty(x)=0,
\]
where $d(x)\in\partial f(x)$. Note that $s(x)=-H(x)v(x)$, thus $s(x)=d(x)-A^Ty(x)$. By complementarity condition, $s_{\cJ(x)}\geq0$. Moreover, by the equivalence of stable set, we have 
\[
x\circ s(x)=x\circ(d(x)-A^Ty(x))=v(x)=0.
\]
Thus, $x_{\cI(x)}>0$ implies $s_{\cI(x)}(x)=0$. Thus, $s(x)\in-N_{\overline{C}}(x)$. Therefore, $0\in\partial f(x)+N_L+N_{\overline{C}}$. 

Next, we prove (i). Without loss of generality, assume $\bar x=0$, otherwise, we only consider the indices set $\cJ(\bar x)$. For any point $\bar x$ in the stable set such that $s(\bar x)<0$, we have there exists $\epsilon>0$ and a neighborhood of $\bar x$: $\cN=(\bar x,\bar x+\epsilon\mathbf{1})$, such that for any $x\in\cN\cap\cM\cap C$, there exists an $\delta>0$, such that $s(x)\leq-\delta<0$. The differential inclusion $\dot{x}=-v(x)\in-{\rm P}_x(H(x)^{-1}\partial f(x))$ implies $\dot x>\delta>0$ as long as $x(t)$ lie in the neighborhood $\cN\cap\cM\cap C$ of $\bar x$. 

We divide by two cases. We first show that if $x(0)\in\cN^c\cap\cM\cap C$, it will enter $\cN\cap\cM\cap C$. Otherwise,  let $t_0:=\inf\{t:x(t)\in\cN\cap\cM\cap C\}$, we have $t_0<\infty$. By continuity, $x(t_0)\in\partial\cN\cap\cM\cap C$. Since $x(t)$ entering $\cN\cap\cM\cap C$, there exists $t_2>t_1>t_0$, such that $x(t_1)_{i_0}>x(t_2)_{i_0}$ for some $i_)$. Note that 
\[x(t_2)_{i_0}=x(t_1)_{i_0}+\int_{t_1}^{t_2}\dot x(s)_{i_0}ds\geq x(t_1)_{i_0}+\delta(t_2-t_1),
\] 
which leads to a contradiction. In this case, $\bar x$ is not a limit point of $x(t)$. Second, if $x(t_0)\in\cN\cap\cM\cap C$, we show that $x(t)$ will exit the neighborhood and will return. Since $x(t)=x(t_0)+\int_{t_0}^t\dot x(s) ds\geq x(t_0)+\delta(t-t_0)>0$, thus, as long as $x(t)\in\cN\cap\cM\cap C$, $x(t)$ increases linearly to time $t$. $x(t)$ will leave $\cN\cap\cM\cap C$, otherwise, we have 
\[
x\qty(t_0+\frac{2\epsilon-\norm{x(t_0)}_\infty}{\delta})\geq x(t_0)+\delta\frac{2\epsilon-\norm{x(t_0)}_\infty}{\delta}\geq2\epsilon.
\]
It is contradictory. Once, $x(t)$ leaves the neighborhood, we can argue in case 1. Thus, $x(t)$ will never subsequentially converge to such spurious stationsrt point.
\end{proof}

Finally, we make extension to \textbf{nonlinear} constraint case. The complementarity condition can be extended to nonlinear constraints, where $C = \{x \in \mathbb{R}^n : g(x) < 0\}$, $\overline{C} = \{x \in \mathbb{R}^n : g(x) \leq 0\}$, and $\mathcal{M} = \{x \in \mathbb{R}^n : c(x) = 0\}$, assuming the Linear Independence Constraint Qualification (LICQ) holds everywhere on $\mathcal{M}$. Define:
$$\mathcal{J}(x) = \{j : g_j(x) = 0\}, \quad \mathcal{I}(x) = \{i : g_i(x) < 0\}.$$
Under the conditions of Proposition \ref{prop:equiv_stable}, the stable set satisfies:
$$0 \in \nabla^2 \phi(x)^{-1} (\partial f(x) + N_x \mathcal{M}) \quad \Longleftrightarrow \quad 0 \in \partial f(x) + N_x \mathcal{M} + \text{span}(N_C(x)).$$
The moving direction $v(x)$ is defined as the solution to:
$$\min_{v \in T_x \mathcal{M}} \left\{ \langle d(x), v \rangle + \frac{1}{2} \|v\|_x^2 \right\},$$
where $T_x \mathcal{M} = \{u : A_x u = 0\}$ for linear map $A_x=\nabla c(x)^T$. The projection onto $T_x \mathcal{M}$ is:
$${\rm P}_x = I - H(x)^{-1} A_x^T (A_x H(x)^{-1} A_x^T)^{-1} A_x,$$
with $v(x) = -{\rm P}_x H(x)^{-1} d(x)$ and dual variable $y(x) = (A_x H(x)^{-1} A_x^T)^{-1} A_x H(x)^{-1} d(x)$. The slack variable is $s(x) = -H(x) v(x) = d(x) - \nabla c(x) \lambda(x)$. {If the complementarity condition $s(x) + x \geq 0$ holds for the $\omega$-limit set, then $s(x) \in -N_{\overline{C}}(x)$: has some issues}, ensuring that limit points are stationary points. Conversely, if the complementarity is strictly violated, the repelling mechanism ensures that $x(t)$ does not subsequentially converge to such spurious stationary points, analogous to the linear case.

\subsection{Avoidance under isolated condition}
In this subsection, we demonstrate that if the $\omega$-limit set (contained in $\cS$) contains only isolated points with respect to $\mathcal{M} \cap \partial C$, then the trajectory converges to a true stationary point in $\Omega$.
\begin{theorem}
\label{thm:isolated}
Suppose Assumption \ref{assumption:separable_phi} holds, if the $\cS$ only has isolated points with respect to $\cM\cap\partial C$, 
then $\{x(t)\}$ converges to a stationary point in $\Omega$.   
\end{theorem}
\begin{proof}
Note that $\cS$ is a closed set, so we have $\lim_{t\rightarrow\infty}{\rm dist}(x(t),\cS)=0$. Then, there are two cases. Case 1: Any limit point of $\{x(t):t\geq0\}$ lies at the interior $C$. Since any point in $\cS\cap C$ is the stationary point, it is done. Case 2: There exists $\{x(t_j)\}$ such that $\lim_{j\rightarrow\infty}t_j=\infty$,  $\lim_{j\rightarrow\infty}x(t_j)=\bar x$, and $\bar x\in\cS\cap\partial C$. We prove that $\lim_{t\rightarrow\infty}x(t)=\bar x$. Otherwise, there exists another distinct limit point of $\{x(t)\}$ denoted as $\bar x'\in\cS$. Since ${\rm dist}(\bar x,\omega(x_0)\setminus\{\bar x\})=:\alpha>0$. Choose the open ball $U=\mathbb{B}^\circ(\bar x,\frac{\alpha}{4})$ and the open set $V=\{x:{\rm dist}(x,\omega(x_0)\setminus\{\bar x\})<\frac{\alpha}{4}\}$. Then $U\cap V=\empty$ and $\omega(x_0)\subset U\cup V$. This is contradict to the fact that $\omega(x_0)$ is a connected set. So in Case two, we prove that $\lim_{t\rightarrow\infty}x(t)=\bar x$ for some $\bar x\in\cS\cap\partial C$.  By Proposition \ref{prop:exiting}, once $x(t)$ converges, it will converge to the stationary point. This complete the proof.
\end{proof}


When the isolated condition does not hold, we propose a random perturbation strategy, in which a novel noise adapted to the kernel function is imposed. We first consider the $\cC^2(\R^n)$ function $f$, then the result is extended to the stratifiable function \cite{bolte2007clarke,davis2020stochastic}.
\begin{lemma}
\label{le:perturb}
Suppose $f$ is a $C^2$ function, define $f_v(x):=f(x)+\inner{\nabla\phi(x),v}$ and $c_u(x)=c(x)+u$. Then for any $\epsilon>0$, almost all $(u,v)\in\epsilon(\mathbb{B}^m\times\mathbb{B}^n)$, the stable set for the perturbed problem: $S(u,v)=\{x:0\in [\nabla^2\phi(x)]^{-1} (\nabla f_v(x)+N_{\mathcal{M}_u}(x))\}$ has no cluster points. 
\end{lemma}
\begin{proof}
For any $x\in S(u,v)$, 
\[
\begin{aligned}
0\in [\nabla^2\phi(x)]^{-1}  (\nabla f_v(x)+N_{\mathcal{M}_u}(x)),
\end{aligned}
\Longleftrightarrow 
\left\{
\begin{aligned}
0\in [\nabla^2\phi(x)]^{-1}  (\nabla f(x)+N_{\mathcal{M}}(x))+v,\\
c(x)+u=0.
\end{aligned}
\right.
\]
Thus, $x\in S(u,v)$ is equivalent to the following nonlinear system in terms of $(x,\lambda)$ having a solution 
\[
\left\{
\begin{aligned}
\nabla^2\phi(x)^{-1}  (\nabla f(x)+\nabla c(x)\lambda)+v&=0\\
c(x)+u&=0.
\end{aligned}
\right.
\]
Let $F(x,\lambda):=\left[
\begin{array}{c}
   [\nabla^2\phi(x)]^{-1} (\nabla f(x)+\nabla c(x)\lambda)  \\
   c(x)  
\end{array}
\right]$. Then the nonlinear system is equivalent to $F(x,\lambda)=-[v;u]$. Choose $[u,v]$ as the regular value of $-F(x,\lambda)$, by Morse-Sard theorem, we have the regular value is dense in $\R^{m}\times\R^n$. Moreover, $\frac{\partial F(x,\lambda)}{\partial (x,\lambda)}$ is nonsingular when $(u,v)$ is a regular value of $F$. By implicit function theorem, for any solution $(\bar x,\bar\lambda)$ of $F(x,\lambda)=-[v;u]$, there is a open neighborhood of $(\bar x,\bar\lambda)$ such that there is only unique solution of the nonlinear system in this neighborhood. Therefore, there is no cluster point in the set $S(u,v)$.
\end{proof}

\begin{theorem}
Suppose $f+\delta_{\cM_u}$ is a stratifiable function for any $u\in\R^m$, where $\cM_u=\qty{x:c(x)+u=0)}$. Define $f_v(x):=f(x)+\inner{\nabla\phi(x),v}$. Then for almost all $(u,v)\in\R^m\times\R^n$, the stable set $S(u,v):=\{x:0\in \nabla^2\phi(x)^{-1} (\partial f_v(x)+N_{\mathcal{M}_u}(x)),x\in\cM_{u}\}$ has no cluster points. 
\end{theorem}
\begin{proof}
By definition, $S(u,v)=\{x:0\in [\nabla^2\phi(x)]^{-1}  (\partial f(x)+N_{\mathcal{M}}(x))+v,c(x)+u=0\}$. By \cite[Proposition 4]{bolte2007clarke}, $\partial f(x)\subset\nabla_{\cM_x}f(x)+N_{\cM_x}(x)$, where $\cM_x$ is the stratum containing $x$, and there exists finite strata $\cup_{i=1}^I\cM_i=\cM$. Thus, we have
\[
S(u,v)\subset\cup_{i=1}^I\qty{x:0\in \nabla^2\phi(x)^{-1}(\nabla_{\cM_i}f(x)+N_{\cM_i}(x)+v),c(x)+u=0}=:\cup_{i=1}^IS_i(u,v).
\]
Next, we prove that each $S_i(u,v)$ has no cluster point. For any $x\in S_i(u,v)$, it is equivalent to the following nonlinear system has a solution
\[
\left\{
\begin{aligned}
\nabla^2\phi(x)^{-1}(\nabla_{\cM_i} f(x)+\nabla c_i(x)\mu_i)+v&=0\\
c_i(x)+u&=0.
\end{aligned}
\right.
\]
Then, the Morse-Sard theorem for the smooth case can be applied. The remaining of the proof is similar to the proof of Lemma \ref{le:perturb}, we omit it.
\end{proof}

This theorem illustrates that introducing a random perturbation ensures that the subgradient flow of the perturbed problem almost surely possesses an isolated stable set. By Theorem \ref{thm:isolated}, the trajectory converges to a stationary point of the perturbed problem.

In the case of linear constraints, the perturbed problem generates solutions that approximate the original problem by relaxing primal feasibility and complementarity conditions. The stable set of the differential inclusion for the perturbed problem is defined as:
\[
\left\{\begin{aligned}
    -v&\in x\circ\qty(\partial f(x)+A^Ty)\\
    -u&=Ax-b.
\end{aligned}\right.
\]
Let $s$ be the slack variable, such that $s\in\partial f(x)+A^Ty$, the system can be rewritten as:
\[
\left\{\begin{aligned}
    &Ax-b=-u\\
    &s\in\partial f(x)+A^Ty\\
    &x\circ s=-v.
\end{aligned}\right.
\]
Typically, $ v $ is selected as a small negative vector, and $ s $ serves as the dual slack variable.

For nonlinear constraints, the perturbed system is expressed as:
\[
\left\{
\begin{aligned}
0\in\;&\nabla^2\phi(x)^{-1}(\partial f(x)+N_x\cM)+v\\
0=&\;c(x)+u.
\end{aligned}\right.
\]
If $ x $ lies on the boundary, $ \nabla^2 \phi(x)^{-1} $ becomes degenerate, and for almost all $ v \in \mathbb{R}^n $, the inclusion $ 0 \in \nabla^2 \phi(x)^{-1} \left( \partial f(x) + N_x \mathcal{M} \right) + v $ has no solution. Consequently, $ x $ must reside in $ \mathcal{M} \cap C $, where $ \nabla^2 \phi(x)^{-1} $ is positive definite. Define $ s(x) \in \partial f(x) + N_x \mathcal{M} $ such that $ s(x) = -\nabla^2 \phi(x) v $. In the case of conic programming, choosing $ v $ appropriately, $ s(x) $ can be ensured to lie in the dual cone. Besides the example of the nonnegative orthant cone, consider the positive definite cone and its kernel function:
\[
\phi(X)=-\log({\rm det}(X)),\; C=\mathbb{S}^d_{++}=\qty{X\in\R^{d\times d}:\;g(X)=-\det(X)<0}.
\]
Since $s(X)V=-X^{-1}VX^{-1}$, selecting $ V \in \mathbb{S}_{--}^d $ ensures $ s(X) = -\nabla^2 \phi(X) V \in \mathbb{S}_{++}^d $, satisfying the dual cone condition. The approximation error is bounded by:
$$\| s(X) \circ X \| \leq \| \text{Diag}(X) \nabla^2 \phi(X) V \| \leq \epsilon \| \text{Diag}(X) \nabla^2 \phi(X) \|.$$
Thus, the perturbed solution yields an approximate stationary point with controlled errors in primal feasibility and complementarity.

Finally, we make some remarks on the application of the stable manifold theorem. The stable manifold theorem \cite{shub2013global} is a powerful tool used to establish saddle point avoidance properties for optimization algorithms and has recently gained significant interest in the optimization community \cite{lee2016gradient,panageas2016gradient,lee2019first,panageas2019first,davis2021subgradient,hsieh2023riemannian}. These studies demonstrate that, under random initialization, the sequences generated by iterative algorithms almost surely avoid converging to strict saddle points, where the magnitude of the updating map is greater than one. However, these findings do not imply the subsequential convergence of the algorithm, such as the properties of the \(\omega\)-limit set. In contrast, our work focuses on investigating the limit point behavior, thus distinguishing our results from those centered on saddle point avoidance.

\section{Iterative algorithms}
In this section, we propose the iterative algorithms such that the generated sequence is the interpolated process of the continuous trajectory. Thus, the convergence properties of the iterative algorithms can be obtained from the continuous trajectory's properties.

\subsection{Riemannian Hessian barrier methods}
The following algorithm is the a discretization of the continuous differential inclusion \eqref{eq:DI_manifold}:
\begin{equation}
\label{eq:alg-RHBA}
x_{k+1}={\rm R}_{x_k}(-\eta_k{\rm P}_{T_{x_k}\cM}\nabla^2\phi(x_k)^{-1}(d_k+\xi_k)),\;x_0\in\cM\cap C,
\end{equation}
where $d_k\in\partial^{\delta_k}f(x_k)$, $\xi_k$ is the stochastic noise, $ {\rm P}_{T_{x_k} \mathcal{M}} $ is the projection onto the tangent space $ T_{x_k} \mathcal{M}$, $ \eta_k $ is the step size, and ${\rm R}$ is a retraction map. The update direction is directly inherent from the right hand side of the DI. This update direction is derived from the continuous DI. To ensure all iterates $ x_k \in \mathcal{M} \cap C $, a retraction map $ {\rm R}_x $ is used, making this an interior point method where iterates stay in the relative interior of the feasible set. To ensure the iterative sequence approximates the trajectory of the continuous DI, we impose the following assumptions:
\begin{assumption}
\label{assumption:alg}
    \begin{enumerate}
        \item {$\phi$ is a Legendre kernel function over $\R^n$}, $\mu$-strongly convex, and $\nabla\phi$ is differentiable almost everywhere. Here $\mu>0$.
        \item The sequences $\{x_k\}$ and $\{\nabla^2\phi(x_k)^{-1}\}$ are uniformly bounded.
        \item The stepsize $\{\eta_k\}$ satisfies $\sum_{k=0}^\infty\eta_k=\infty$ and $\eta_k=o\qty(\frac{1}{\log k})$.
        \item The noise is a martingale difference noise sequence $\{\xi_k\}$,  i.e. $\E[\xi_{k+1}|\cF_k]=0$, and uniformly bounded. 
               
        \item The set-valued mapping $\mathcal{H}$ has a closed graph. Additionally, for any unbounded increasing sequence $\{k_j\}$ such that $\{x_{k_j}\}$ converges to $\bar x$, it holds that 
        \begin{equation}
            \lim_{N\rightarrow\infty}{\rm dist}\left(\frac{1}{N}\sum_{j=1}^Nd_{k_j},\mathcal{H}(\bar x)\right)=0.
        \end{equation}
        \item For any $\rho\in(0,\infty)$,
        \[
        \lim_{\eta\to0}\sup_{x\in\cM\cap C\cap\mathbb{B}_\rho}\sup_{d\in\mathbb{B}_{\rho}}\frac{\norm{{\rm R}_x(-\eta{\rm P}_{T_x\cM}\nabla^2\phi(x)^{-1}d)-\qty(x-\eta{\rm P}_{T_x\cM}\nabla^2\phi(x)^{-1}d)}}{\eta}=0.
        \]
    \end{enumerate}
\end{assumption}
We first make some comments on Assumption \ref{assumption:alg}.1-5, which are standard assumptions in stochastic approximation approach applied in nonsmooth nonconvex optimization. For the novel assumption in our algorithm is Assumption \ref{assumption:alg}.6. We make a separate discussion later.
\begin{remark}
    \label{rmk:assumption_DI}
    \begin{enumerate}
        \item If the Legendre function $\phi$ is supercoecive (implied by strong convexity), i.e. $\lim_{\|u\|\rightarrow\infty}\frac{\phi(u)}{\|u\|}=\infty$, then by \cite[Theorem 26.5,\,Corollary 13.3.1]{rockafellar1997convex}, $\phi^*\in\cC^1(\R^n)$ is strictly convex, and $(\nabla\phi)^{-1}=\nabla\phi^*$. Once restricted on a compact set, it is not restrictive that $\phi$ is $\mu$-strongly convex. For example, consider entropy function $\phi(x)=x\log(x)-x$ over simplex $\{x\in\R^n:xe=1,x\geq0\}$, which is a compact set. Although $\phi$ is not strongly convex over entire $\R^n$, $\phi$ is $1$-strongly convex.
        \item The uniform boundedness assumption is directly assumed in many existing works (e.g. \cite{davis2020stochastic,castera2021inertial,xiao2023adam}) on stochastic subgradient methods, since the compact condition is essential to apply the technique of stochastic approximation \cite{benaim2005stochastic,borkar2009stochastic} to establish the convergence. Recent works \cite{josz2024global,xiao2023convergence} prove the global stability of the iterates of certain subgradient methods for coercive definable function with random reshuffling technique. Here, we also directly to assume the boundedness condition on $\{x_k\}$. 
        \item The bounded condition for $\nabla^2\phi(x_k)^{-1}$ is mild given the boundedness of $\{x_k\}$. Since $\phi(x)$ diverges at the domain’s boundary, $\nabla^2 \phi(x)^{-1}$ is generally well-controlled and even flattens near the boundary. For instance, with the entropy function $\phi(x) = x \log x - x$ over $C = \{x : x \geq 0\}$, we have $\nabla^2 \phi(x)^{-1} = {\rm Diag}(x)$, which remains bounded on compact sets.

        \item 
        Uniformly bounded martingale difference noise is a standard assumption in stochastic optimization, as noted in \cite{castera2021inertial,xiao2023adam}, especially for finite-sum objectives with noise from random reshuffling. With stepsizes chosen in Assumption \ref{assumption:alg}.3, \cite{benaim2005stochastic} shows that:
        \begin{equation}
            \lim_{s\rightarrow\infty}\sup_{s\leq i\leq\Lambda_\eta(\lambda_\eta(s)+T)}\norm{\sum_{k=s}^i\eta_k\xi_k}=0,
            \label{Eq:noise_cond}
        \end{equation}   
        indicating that accumulated noise over fixed time windows is manageable.
        \item Assumption 1.5 requires $d_k$ to approximate the set-valued mapping $\mathcal{H}(x_k)$. This is a mild condition, satisfied, for example, if $d_k \in \mathcal{H}^{\delta_k}(x_k)$ with $\delta_k \to 0$.
    \end{enumerate}
\end{remark}

Now, we make some discussions on Assumption \ref{assumption:alg}.6. In stochastic approximation theory, the step sizes $\eta_k$ must satisfy two key conditions: $\sum_{k}\eta_k=\infty$ and $\lim_{k\to\infty}\eta_k=0$, in which the unsummable condition can ensure the sequence can travel arbitrarily far. Moreover, the discrete sequence should approximate the continuous trajectory sufficiently well. Specifically, the scheme should satisfy $x_{k+1}=x_k-\eta_k(v_k+\delta_k)$, where $v_k$ is the update direction and $\delta_k \to 0$ as $k \to \infty$. However, in Riemannian optimization, especially when iterates approach the boundary of the feasible set, these conditions can be challenging to satisfy. The retraction map $R$, which enforces feasibility within the manifold $\mathcal{M} \cap C$, may disrupt the accurate approximation of the continuous trajectory.

To illustrate it more clearly, consider optimization over the interval $[0,1] \subset \mathbb{R}$ with a constant update direction $+1$ and a Euclidean metric, where $\nabla^2 \phi(x) = \text{Id}$. A typical Riemannian update gives:
\[
x_{k+1}={\rm R}_{x_k}(-\eta_k\cdot1)=x_k-\eta_k\qty(1+\delta_k),\text{ where $\lim_{k\to\infty}\delta_k=0$.}
\]
However, such Retraction never exists. Note that 
\[
x_{k}=x_0-\sum_{i=0}^{k-1}\eta_i(1+\delta_i)\to-\infty,
\]
since $\sum_{k} \eta_k = \infty$, $x_k \to -\infty$, which violates the feasibility constraint $x_k \in (0,1)$. No retraction can simultaneously maintain both the accurate approximation of the continuous trajectory and strict feasibility in this scenario. 
 
Thus, Assumption \ref{assumption:alg}.6 basically requires the existence of the retraction to ensure both the well-approximation to the continuous trajectory and strict feasible. The magic of our proposed algorithm lies in its ability to ensure the existence of a suitable retraction map, thanks to the metric induced by the Hessian of carefully chosen barrier function. As iterates $x_k$ near the boundary, the inverse Hessian $\nabla^2 \phi(x_k)^{-1}$ shrinks the update direction $v_k = {\rm P}_{T_{x_k} \mathcal{M}} \nabla^2 \phi(x_k)^{-1} (d_k + \xi_k)$, preventing the iterates from leaving the feasible set, and thus reducing the need for sharp step size adjustments.

To illustrate the subtlety of our approach, we consider the case of an affine manifold with the retraction map chosen as the identity map. We introduce self-concordant logarithmically homogeneous barrier functions, which ensure that our proposed algorithm effectively approximates the continuous trajectory while keeping all iterates within the relative interior of the feasible set.

Let \(\bar{K} \subset \mathbb{R}^n\) be a \textit{regular cone}: \(\bar{K}\) is closed convex with nonempty interior \(K = \mathrm{int}(\bar{K})\), and pointed (i.e., \(\bar{K} \cap (-\bar{K}) = \{0\}\)). Any such cone admits a self-concordant logarithmically homogeneous barrier kernel \(\phi(x)\) with finite parameter value \(\theta\).
\begin{definition}
A function \(\phi : \bar{K} \to (-\infty, \infty]\) with \(\mathrm{dom} \, \phi = K\) is called a \textit{\(\theta\)-logarithmically homogeneous self-concordant barrier} (\(\theta\)-LHSCB) for the cone \(\bar{K}\) if:

(a) \(\phi\) is a \(\theta\)-self-concordant barrier for \(\bar{K}\), i.e., for all \(x \in K\) and \(u \in \mathbb{R}^n\)
\[
|D^3 \phi(x)[u, u, u]| \leq 2D^2 \phi(x)[u, u]^{3/2}, \quad \text{and} \quad \sup_{u \in \mathbb{E}} |2D \phi(x)[u] - D^2 \phi(x)[u, u]| \leq \theta.
\]

(b) \(\phi\) is \textit{logarithmically homogeneous}:
\[
\phi(tx) = \phi(x) - \theta \ln(t) \quad \forall x \in K, \, t > 0.
\]
\end{definition}

We now present a proposition that guarantees the iterates remain in the relative interior of the feasible set under suitable step size constraints.
\begin{proposition}
\label{prop:stepsize_threshhold}
Suppose $\sup_{k}\qty{\norm{d_k}}\leq M_d$, $\sup_{k}\norm{\xi_k}\leq M_\xi$ and $\sup_{k}\norm{\nabla^2\phi(x_k)^{-1}}\leq\hat M$. Suppose $\cM$ is affine, $\bar C$ is a regular cone, and $\phi$ is a $\theta$-logarithmically homogeneous self-concordant barrier for cone $\bar C$, then as long as $x_k\in\cM\cap C$, $\eta_k\in\left[0,\frac{1}{\hat M\qty(M_d+M_\xi)^2}\right]$, we have 
\[
x_{k+1}={\rm R}_{x_k}(-\eta_k{\rm P}_{T_{x_k}\cM}\nabla^2\phi(x_k)^{-1}(d_k+\xi_k))=x_k-\eta_k{\rm P}_{T_{x_k}\cM}\nabla^2\phi(x_k)^{-1}(d_k+\xi_k)\in\cM\cap C.
\]
\end{proposition}
\begin{proof}
Since $x_k\in\cM$, $\mathcal{M}$ is an affine manifold, and $P_{T_{x_k} \mathcal{M}} = P_{\mathcal{M}}$, it follows that $x_{k+1} \in \mathcal{M}$. We need only verify that $x_{k+1} \in C$. Consider the norm in the local metric:
\[
\begin{aligned}
&\norm{x_{k+1}-x_k}_{x_k}=\norm{-\eta_k{\rm P}_{T_{x_k}\cM}\nabla^2\phi(x_k)^{-1}(d_k+\xi_k)}_{x_k}\leq\eta_k\norm{ \nabla^2\phi(x_k)^{-1}(d_k+\xi_k)}_{x_k}\\
=&\eta_k\inner{\nabla^2\phi(x_k)^{-1}(d_k+\xi_k),(d_k+\xi_k)}\leq\eta_k \hat M\qty(M_d+M_\xi)^2\leq1.
\end{aligned}
\]
Since $\phi$ is a $\theta$-logarithmically homogeneous self-concordant barrier for cone $\bar C$, we know from \cite[Theorem 5.1.5]{nesterov2018lectures} that Dikin ellipsoid $\mathcal{W}(x,1)\subset C$. Thus, $x_{k+1}\in C$. Thus, $x_{k+1}\in\cM\cap C$.
\end{proof}
The proposition establishes a safe step size threshold $\bar{\eta}$. This allows us to select step sizes such that $\sup_k \{\eta_k\} \leq \bar{\eta}$ and $\sum_k \eta_k = \infty$, ensuring feasibility while permitting sufficient exploration of the space for convergence. To guarantee convergence of the iterates, additional assumptions on $f$ are needed. 
\begin{assumption}    \label{assumption:Lyapunov_RHBA}
    \begin{enumerate}
        \item $f$ is lower bounded, i.e. $\liminf_{x\in\R^n}f(x)>-\infty$. Moreover, $f$ is path-differentiable.
        \item The critical value set $\{f(x):\;0\in \nabla^2\phi(x)^{-1}(\partial f(x)+N_{x}\cM)\}$ has empty interior in $\R$.
    \end{enumerate}
\end{assumption}
Assumption \ref{assumption:Lyapunov_RHBA} is a specific instance of Assumption \ref{assumption_Sard_Lyapunov}. In particular, $f$ is the Lyapunov function and Assumption \ref{assumption:Lyapunov_RHBA}(2) requires a weak Sard condition for \(f\). Recall that the classical Sard theorem for differentiable functions in \(\mathbb{R}^n\) states that the set of critical values has Lebesgue measure zero.  
This assumption is not restrictive. Suppose conditions in Proposition \ref{prop:relation_NC_Null} holds, the critical values set becomes
\[
\{f(x):\;0\in\partial f(x)+N_{x}\cM+{\rm span}(N_{\overline{C}}(x))\}
\]
Suppose $f$ and $\overline{C}$ are stratifiable, and $\overline{C}=\cup_{i=1}^IS_i$, where every $S_i$ are smooth disjoint manifolds. Then, we have
\[
\{f(x):\;0\in\partial f(x)+N_{x}\cM+{\rm span}(N_{\overline{C}}(x))\}\subset\cup_{i=1}^I\{f(x):\;0\in\partial f(x)+N_{x}\cM+N_{x}S_i\}.
\]
By \cite{bolte2007clarke}, Sard theorem holds for each $\{f(x):\;0\in\partial f(x)+N_{x}\cM+N_{x}S_i\}$, thus Assumption \ref{assumption:Lyapunov_RHBA}.2 holds. 

Note that \eqref{assumption:alg} can be rewritten as the form of general subgradient methods:
\begin{equation}
\label{eq:alg-RHBA-reform}
\begin{aligned}
x_{k+1}=&x_k-\eta_k{\rm P}_{T_{x_k}\cM}\nabla^2\phi(x_k)^{-1}(d_k+\xi_k)+ \eta_k\Delta_k\\
=&x_k-\eta_k\qty({\rm P}_{T_{x_k}\cM}\nabla^2\phi(x_k)^{-1}d_k+{\rm P}_{T_{x_k}\cM}\nabla^2\phi(x_k)^{-1}\xi_k)+\eta_k\Delta_k,\\
=&x_k=-\eta_k(v_k+\tilde\xi_k).
\end{aligned}
\end{equation}
where 
\[
\Delta_k=\frac{{\rm R}_{x_k}(-\eta_k{\rm P}_{T_{x_k}\cM}\nabla^2\phi(x_k)^{-1}(d_k+\xi_k))-\qty(x_k-\eta_k{\rm P}_{T_{x_k}\cM}\nabla^2\phi(x_k)^{-1}(d_k+\xi_k))}{\eta_k},
\]
and $v_k={\rm P}_{T_{x_k}\cM}\nabla^2\phi(x_k)^{-1}d_k+\Delta_k$.
Thus, the convergence of the algorithm can be established using Theorem 1, provided that all necessary assumptions are verified. We begin by checking the noise condition.

\begin{lemma}
\label{le:noise_ass}
Let $\tilde\xi_k:={\rm P}_{T_{x_k}\cM}\nabla^2\phi(x_k)^{-1}\xi_k$, then $\E[\tilde\xi_k|\cF_{k-1}]=0$, and $\{\tilde\xi_k\}$ is uniformly bounded.
\end{lemma}
\begin{proof}
$\E[\tilde\xi_k|\cF_{k-1}]=0$ is directly from definition. Moreover, it holds that
\[
\norm{{\rm P}_{T_{x_k}\cM}\nabla^2\phi(x_k)^{-1}\xi_k}_x^2\leq\norm{\nabla^2\phi(x_k)^{-1}\xi_k}_x^2=\inner{\nabla^2\phi(x_k)^{-1}\xi_k,\xi_k}\leq \hat MM_\xi^2<\infty.
\]
Thus, we have
\[
\begin{aligned}
\norm{{\rm P}_{T_{x_k}\cM}\nabla^2\phi(x_k)^{-1}\xi_k}^2\leq\frac{1}{\mu}\norm{{\rm P}_{T_{x_k}\cM}\nabla^2\phi(x_k)^{-1}\xi_k}_x^2\leq\frac{\hat MM_\xi^2}{\mu}<\infty.
\end{aligned}
\]
This completes the proof.
\end{proof}

Next, we verify the set-valued approximation condition. 

\begin{lemma}
\label{le:d_ass}
Suppose Assumption \ref{assumption:alg} holds. Then 
for any increasing sequence $\{k_j\}$ such that $\{x_{k_j}\}$ converges to $\bar x$, it holds that 
    \[  \lim_{N\rightarrow\infty}\dist\left(\frac{1}{N}\sum_{j=1}^Nv_{k_j},{\rm P}_{T_{\bar x}\cM}\nabla^2\phi(\bar x)^{-1}\partial f(\bar x)  \right)=0.
    \]
\end{lemma}
\begin{proof}
Recall that $v_k={\rm P}_{T_{x_k}\cM}\nabla^2\phi(x_k)^{-1}d_k+\Delta_k$. By Assumption \ref{assumption:alg}.6, we have $\lim_{k\to\infty}\norm{\Delta_k}=0$. Since $\nabla^2\phi^{-1}$ is continuous, ${\rm P}_{T_{x}\cM}\nabla^2\phi(x)^{-1}$ is a continuous map. Note that $\partial f$ is outer-semicontinuous, and ${\rm P}_{T_{\bar x}\cM}\nabla^2\phi(\bar x)^{-1}\partial f(\bar x)$ is a convex set, thus 
\[
\begin{aligned}
&\dist\left(\frac{1}{N}\sum_{j=1}^Nv_{k_j},{\rm P}_{T_{\bar x}\cM}\nabla^2\phi(\bar x)^{-1}\partial f(\bar x)  \right)\leq\frac{1}{N}\sum_{j=1}^N \dist\left(v_{k_j},{\rm P}_{T_{\bar x}\cM}\nabla^2\phi(\bar x)^{-1}\partial f(\bar x)  \right)\\
\leq&\frac{1}{N}\sum_{j=1}^N \dist\left({\rm P}_{T_{ x_{k_j}}\cM}\nabla^2\phi( x_{k_j})^{-1}\partial f(\bar x),{\rm P}_{T_{\bar x}\cM}\nabla^2\phi(\bar x)^{-1}\partial f(\bar x)  \right)\\
&+\frac{1}{N}\sum_{j=1}^N \dist\left({\rm P}_{T_{ x_{k_j}}\cM}\nabla^2\phi( x_{k_j})^{-1}\partial f(x_{k_j}),{\rm P}_{T_{ x_{k_j}}\cM}\nabla^2\phi( x_{k_j})^{-1}\partial f(\bar x)  \right)\longrightarrow0.
\end{aligned}
\]
This completes the proof.
\end{proof}

\begin{proposition}
\label{prop:Lyapunov_RHBA}
Suppose Assumptions \ref{assumption:alg} and \ref{assumption:Lyapunov_RHBA} hold. Then $f$ is a Lyapunov function for the differential inclusion \eqref{eq:DI_manifold} with the stable set $\left\{x\in\R^n:\;0\in{\rm P}_{T_x\cM}(\nabla^2\phi(x)^{-1}\partial f(x))\right\}$.
\end{proposition}
\begin{proof}
Consider any trajectory $z(\cdot)$ for the differential inclusion \eqref{eq:DI_manifold} with $0\notin{\rm P}_{T_x\cM}(\nabla^2\phi(z(0))^{-1}\partial f(z(0)))$. By the equivalent reformulation \eqref{eq:DI_manifold_Breg}, we have
\[
\frac{d}{dt}\nabla\phi(x(t))\in-\partial f(x)-N_x\cM.
\]
Note that $\dot z\in T_x\cM$, inner product with $\dot z$ for the above inclusion, we have
\[
\begin{aligned}
\frac{d}{ds}f(z(s))=\inner{\partial f(z(s)),\dot z(s)}\ni-\inner{\nabla^2\phi(z(s))\dot z(s),\dot z(s)}.
\end{aligned}
\]
Therefore, for any $t\geq0$, it holds that
\[
f(z(t))-f(z(0))=-\int_0^t\inner{\nabla^2\phi(z(s))\dot z(s),\dot z(s)}ds\leq-\int_0^t\mu\norm{\dot z(s)}^2ds\leq0.
\]
We now prove the required result by contradiction. Suppose for any $t\geq0$, $f(z(t))=f(z(0))$, then $\dot z(s)=0$ for almost all $s\geq0$. Since $z(\cdot)$ is absolutely continuous, therefore, $z(t)\equiv z(0)$ for any $t\geq0$. Then, 
\[
0=\dot z(t)\in-{\rm P}_{T_{z(t)}\cM}(\nabla^2\phi(z(t))^{-1}\partial f(z(t)))=-{\rm P}_{T_{z(0)}\cM}(\nabla^2\phi(z(0))^{-1}\partial f(z(0))).
\]
This is contradictory to the fact that $0\notin{\rm P}_{T_x\cM}(\nabla^2\phi(z(0))^{-1}\partial f(z(0)))$. Therefore, there exists $T>0$, such that $f(z(T))<\sup_{t\in[0,T]}f(z(t))\leq f(z(0))$. This completes the proof.
\end{proof}

By Lemma \ref{le:noise_ass}, Lemma \ref{le:d_ass}, Proposition \ref{prop:Lyapunov_RHBA}, and Theorem \ref{convergence-thm-func-val}, we can directly derive the following convergence results for \eqref{eq:alg-RHBA}.
\begin{theorem}
    \label{thm:BGD}
    Suppose Assumptions \ref{assumption:alg} and \ref{assumption:Lyapunov_RHBA} hold. 
    \begin{enumerate}
        \item Then almost surely, any cluster point of $\{x_k\}$ generated by \eqref{eq:alg-RHBA} lies in $\cS$ and the function values $\{f(x_k)\}$ converge. 
        \item Furthermore, if the complementarity condition or isolated condition for the stable set holds, then any limit point is the stationary point of Problem \eqref{prob:PCP}.
    \end{enumerate}
\end{theorem}

As shown in Section 3, random perturbations enable the trajectory to converge to the true stationary points of the perturbed problem. In many instances, the stationary points of the perturbed problem are nearly stationary points of the original problem. This property is inherited by the sequence generated by the iterative algorithm.

\begin{corollary}
    Suppose Assumptions \ref{assumption:alg} and \ref{assumption:Lyapunov_RHBA} hold, given a tolerance $\epsilon>0$. Randomly choose a $v\in\mathbb{B}_\epsilon$ and $u\in\mathbb{B}_\epsilon$, then any sequence generated by \eqref{eq:alg-RHBA} for the perturbed problem subsequentially converges to a stationary point of the perturbed problem.
\end{corollary}

\subsection{Discussion on mirror descent scheme}
In this section, we briefly discuss the mirror descent scheme, where many problems remain open. As shown in \cite{ding2024stochastic}, the Hessian barrier method and mirror descent scheme for unconstrained problems are unified under the same differential inclusion. In this case, the Hessian barrier method is a special preconditioned subgradient method. Using Taylor's expansion and the property $\nabla^2\phi^*(\nabla\phi(x)) = (\nabla^2\phi(x))^{-1}$, we derive:
\[
\begin{aligned}
x_{k+1}=&\nabla\phi^*(\nabla\phi(x_k)-\eta_kd_k)\\
=&x_k-\eta_k\nabla^2\phi(x_k)^{-1}d_k+O(\eta_k^2),
\end{aligned}
\]
Thus, the linear interpolation of the sequence ${x_k}$ represents a perturbed solution of the continuous trajectory and inherits its convergence properties.

The situation for constrained problems is significantly more challenging, particularly regarding boundary behavior. Analogous to the Riemannian Hessian-barrier method, we propose the following heuristic Riemannian mirror descent scheme:
\begin{equation}
\begin{aligned}
\tilde x_{k+1}=&\argmin_{x\in\R^n}\;\left\{\inner{d_k+\xi_k,x}+\frac{1}{\eta_k}\cD_\phi(x,x_k):\;\text{s.t. }x-x_k\in T_{x_k}\cM\right\}\\
x_{k+1}=&{\rm R}_{x_k}(\tilde x_{k+1}-x_k).
\end{aligned}
\label{R-Breg}
\tag{R-Breg}
\end{equation}
where $d_k\in\partial f^{\delta_k}(x_k)$ with $\lim_{k\to\infty}\delta_k=0$. When $\mathcal{M}$ is an affine space $\{ x : Ax = b \}$ and the retraction is the identity map, the algorithm simplifies to:
\begin{equation}
\begin{aligned}
x_{k+1}=\argmin_{x\in\R^n}\;&\left\{\inner{d_k+\xi_k,x}+\frac{1}{\eta_k}\cD_\phi(x,x_k)\right\}\\
\text{s.t. }&Ax=b.
\end{aligned}
\label{Breg}
\end{equation}
Given $x\in\cM\cap C$, $d\in\R^n$ and $\eta>0$, define $(x^+(x,d,\eta),y(x,d,\eta))$ as the solution to:
\[
\left\{
\begin{aligned}
\nabla\phi^*(\nabla\phi(x)-\eta(d-A^Ty))=x^+,\\
A\nabla\phi^*(\nabla\phi(x)-\eta(d-A^Ty))=b.
\end{aligned}\right.
\]
By definition, $x_{k+1}=x^+(x_k,d_k+\xi_k,\eta_k)$. We require the following assumption:
\begin{equation}
\limsup_{\eta\to0^+}\sup_{d\in\mathbb{B}_\rho}\sup_{x\in\mathbb{B}_\rho\cap\cM\cap C}\frac{\norm{x-\eta\nabla^2\phi(x)^{-1}(d-A^Ty(x,d,\eta))-\nabla\phi^*(\nabla\phi(x)-\eta(d-A^Ty(x,d\eta)))}}{\eta\norm{\nabla^2\phi(x)^{-1}g}}=0.
\label{eq:interpolate_Breg}
\end{equation}
Under this assumption, the interpolated process of $\{x_k\}$ is a perturbed solution of the differential inclusion \eqref{eq:DI_manifold}. Consequently, all convergence results from Sections 3 and 4 for the trajectory of the differential inclusion on the manifold are inherited by the Riemannian mirror descent scheme. However, verifying this condition is generally difficult. An example using the entropy function over the negative orthant can satisfy this condition.

\begin{example}
For $\cM=\R^n$, $C=\R^n_{++}$, and the kernel function $\phi(x) = \sum_{i=1}^n (x_i \log x_i - x_i)$, the condition \eqref{eq:interpolate_Breg} holds.
\end{example}
\begin{proof}
When $\phi(x)=\sum_{i=1}^n(x_i\log(x_i)-x_i)$, basic calculation gives
\[
\nabla\phi^*(\nabla\phi(x)-\eta d)=x\circ e^{-\eta d}.
\] 
For a given $\rho > 0$, consider each coordinate:
\[
\begin{aligned}
&\limsup_{\eta\to0^+}\frac{\norm{x_i-\eta x_i d_i-x_i e^{-\eta d_i}}}{\eta}=\limsup_{\eta\to0^+}\frac{\norm{x_i-\eta x_id_i-x_i(1-\eta d_i+\frac{1}{2}(\eta d_i)^2)e^{-\zeta_id_i}}}{\eta}\\
=&\limsup_{\eta\to0^+}\frac{1}{2}\eta d_i^2e^{-\zeta_id_i}=0,
\end{aligned}
\]
where $\zeta_i\in[0,\eta_i]$. This completes the proof.
\end{proof}

Suppose Assumption \ref{assumption:alg} and Assumption \ref{assumption:Lyapunov_RHBA} hold.  The mirror descent scheme for the above example generates a perturbed solution of the corresponding differential inclusion. Consequently, the iterates subsequentially converge to a true stationary point, rather than a spurious one, under an isolation condition or complementary condition. Moreover, the isolation condition typically holds when small random perturbations are applied to the objective function and constraints.

As a closing note for this section, we highlight a key property of the mirror descent scheme applied to optimization problems with linear constraints: under mild conditions, if the sequence of iterates generated by the scheme converges, it converges to a stationary point of the original problem.
\begin{proposition}
\label{prop:Breg_convergence}
Consider an affine space $\mathcal{M}$ and an open set $C$. Suppose the following conditions hold: 
\begin{enumerate}
    \item $\phi$ is a Legendre function over $C$;
    \item Stepsize and noise satisfy 
    \[        \eta_k>0,\quad\sum_{k=0}^\infty\eta_k=\infty,\quad\lim_{s\rightarrow\infty}\sup_{s\leq i\leq\Lambda_\eta(\lambda_\eta(s)+T)}\norm{\sum_{k=s}^i\eta_k\xi_k}=0.
    \]
    \item For any $x\in\cM_{\overline{C}}$: $N_{\overline{C}}(x)\cap N_x\cM=\{0\}$.
\end{enumerate}
If $x_k$ converges to $\bar x$, then $\bar x$ is the stationary point of the original optimization problem.
\end{proposition}
\begin{proof}
The update of mirror descent scheme gives
\[
\nabla\phi(x_{k+1})\in\nabla\phi(x_k)-\eta_k\partial f(x_k)+N_L.
\]
Summing the updates from $k = 0$ to $k-1$, we obtain: 
\[
\frac{\nabla\phi(x_k)-\nabla\phi(x_0)}{\norm{\nabla\phi(x_k)}}\cdot\frac{\norm{\nabla\phi(x_k)}}{\sum_{i=0}^{k-1}\eta_i}\in-\frac{\sum_{i=0}^{k-1}\eta_i\partial f(x_i)}{\sum_{i=0}^{k-1}\eta_i}+N_x\cM,
\]
noting that the sum of the normal cone terms remains $N_L$ because $N_L$ is a cone. Define 
\[
\begin{aligned}
\nu_k=\frac{\nabla\phi(x_k)-\nabla\phi(x_0)}{\norm{\nabla\phi(x_k)}},\;\Gamma_k=\frac{\norm{\nabla\phi(x_k)}}{\sum_{i=0}^{k-1}\eta_i},\;G_k=-\frac{\sum_{i=0}^{k-1}\eta_i\partial f(x_i)}{\sum_{i=0}^{k-1}\eta_i}.
\end{aligned}
\]
Without loss of generality, assume $\nu_k\to\bar\nu$, note that $\norm{\bar\nu}=1$. By outer-semicontinuity of $\partial f$ and $\sum_{k=0}^\infty\eta_k=\infty$, $G_k\to-\partial f(\bar x)$. 
Since $\phi$ is a Legendre kernel, by \cite[Lemma 4.1]{alvarez2004hessian}, $\bar\nu\in N_{\overline{C}}(\bar x)$. When some $\Gamma_{k_j}$ converges to some $\bar\Gamma<\infty$, then 
\[
\bar\Gamma\bar\nu\in-\partial f(\bar x)+N_L.
\]
This is exactly the stationary condition. When $\Gamma_k$ diverges to infinity, we have $\bar\nu\in N_L$. This is contradictory to the fact that $N_{\overline{C}}(x)\cap N_x\cM=\{0\}$. This completes the proof.
\end{proof}
When $\overline{C}=\{x:g_i(x)\leq0\}$, $\cM=\{x:c(x)=0\}$,the condition $N_{\overline{C}}(x) \cap N_x \mathcal{M} = {0}$ is equivalent to the Mangasarian-Fromovitz Constraint Qualification (MFCQ) at $x$. This equivalence is well-established in nonlinear optimization literature, see for example \cite{bertsekas1997nonlinear}.


\section{Conclusion}

We have studied a class of nonsmooth, nonconvex optimization problems over manifolds with boundary, where the feasible region is the intersection of the closure of an open set and a smooth manifold. To address this problem, we introduced a novel continuous-time Riemannian Hessian-barrier Subgradient flow. This dynamical system unifies two prominent optimization paradigms—the Hessian barrier method and the mirror descent scheme, by showing that both arise as discrete-time approximations of the same underlying flow. Our analysis of the trajectory reveals that previously observed convergence deficiencies in the Hessian barrier and mirror descent methods correspond to spurious stationary points of the induced differential inclusion (DI), rather than to true stationary points of the original problem. We further demonstrate that, under mild conditions, trajectories of the continuous flow escape these spurious attractors. By discretizing the RHBS Riemannian Hessian-barrier Subgradient flow via stochastic-approximation techniques, we derive two new interior-point algorithms whose iterates remain strictly within the relative interior of the feasible set. Leveraging the bridge between continuous and discrete dynamics, we prove that—under mild sufficient conditions—every accumulation point of the discrete iterates is a genuine stationary point of the original problem. When those conditions do not hold, we propose a simple perturbation scheme that ensures convergence to a stationary point of a nearby perturbed problem, which is arbitrarily close to a true stationary point of the original formulation.
\bibliography{BG}
\bibliographystyle{abbrv}

\end{document}